\documentclass[12pt]{amsart}
\usepackage{amsmath,amssymb}
\usepackage{eufrak}
\usepackage{amsthm}
\usepackage{enumitem}
\usepackage{graphicx}
\usepackage{graphics}
\usepackage{color}
\usepackage[dvipsnames]{xcolor}
\usepackage{hyperref}
\usepackage{youngtab}
\usepackage{marginnote}
\usepackage{lscape}
\usepackage{tikz}
\usepackage{setspace}  
\usetikzlibrary{calc,shadings,patterns}
\usetikzlibrary{matrix}
\usetikzlibrary{arrows}
\usetikzlibrary{matrix}
\usepackage{verbatim}
\usepackage{multirow}
\usetikzlibrary{decorations.pathreplacing}
\usepackage{enumitem,pifont,xcolor}
\definecolor{myblue}{RGB}{0,29,119}

\newtheorem{theorem}{Theorem}[section]
\newtheorem{proposition}[theorem]{Proposition}

\newtheorem{lemma}[theorem]{Lemma}
\theoremstyle{definition}
\newtheorem{definition}[theorem]{Definition}
\newtheorem{example}[theorem]{Example}

\newtheorem{remark}[theorem]{Remark}

\newtheorem{problem}[theorem]{Problem}

\newtheorem{notation}[theorem]{Notation}
\newtheorem*{theorem*}{Theorem}

\makeatletter
\@addtoreset{equation}{section}
\makeatother

\newcommand{\Cc}{{\mathbb C}} 
\newcommand{\Pp}{{\mathbb P}} 

\newcommand{\cA}{{\mathcal A}}

\newcommand{\cO}{{\mathcal O}}

\providecommand{\AMS}{$\mathcal{A}$\kern-.1667em%
\lower.25em\hbox{$\mathcal{M}$}\kern-.125em$\mathcal{S}$}

\begin{document}
\let\thefootnote\relax

\title[Singularities of Dual Grassmannian]{Singularities of the dual varieties associated to exterior representations:\\\vspace{10pt}1. Dual Grassmannian}

\author{Emre SEN}


\maketitle

\begin{abstract}
For a given irreducible projective variety $X$, the closure of the set of all hyperplanes containing tangents to $X$ is the projectively dual variety $X^{\vee}$. We study the singular locus of projectively dual varieties of certain Segre-Pl\"{u}cker embeddings in series of papers. In this work we give a classification of the irreducible components of the singular locus of the dual Grassmannian. Basically, it admits two components: cusp type and node type which are degeneracies of a certain Hessian matrix, and the closure of the set of tangent planes having more than one critical point, respectively. In particular we reproduce the result about the normality of the dual Grassmannian varieties.
\end{abstract}

\tableofcontents 

\section{Introduction}

Projective duality is one of the fundamental and historical notions in algebraic geometry. Briefly it is the correspondence between lines and hyperplanes in projective space. This relation can be extended to the nonlinear subvarieties of projective space. Namely, for a given irreducible projective variety $X$, the closure of the set of all hyperplanes containing tangents to $X$ is the projectively dual variety $X^{\vee}$.
 
An important class of dual varieties - introduced by I. Gelfand, M. Kapranov and A. Zelevinsky \cite{GKZ} - is hyperdeterminants which are the defining equations of projectively duals of Segre embedding. 
Also they are analogs of determinants for multidimensional matrices. The first treatment of the subject was initiated by A. Cayley and then L. Schl\"{a}fli but it was not studied for the next 150 years. In the early 90's the subject was revitalized by  I. Gelfand, M. Kapranov and A. Zelevinsky. 


In a series of papers  \cite{sen}, the problem we work on is:

\begin{problem}
 Consider the Segre-Pl\"{u}cker embedding :
\begin{align}
X=\Pp\left(\bigwedge\limits^{k_1}\Cc^{N_1}\right)\times\ldots\times\Pp\left(\bigwedge\limits^{k_r}\Cc^{N_r}\right)\rightarrow\Pp\left(\bigwedge\limits^{k_1}\Cc^{N_1}\otimes\cdots\otimes\bigwedge\limits^{k_r}\Cc^{N_r}\right)
\end{align}
where $N_i\geq 2k_i$ for $1\leq i\leq r$. Describe the singular locus of the dual variety $X^{\vee}$ provided that it is a hypersurface, and classify its components.
\end{problem}

We point out that, if all $k_i = 1$, then the Segre-Pl\"{u}cker embedding reduces to the Segre embedding and the dual is a hyperdeterminant hypersurface. In the seminal work, J. Weyman and A. Zelevinsky worked on singularities of hyperdeterminants \cite{WZ}; they classified irreducible components of the singular locus of hyperdeterminants which are cusp and node components. This was one of the questions raised by Gelfand et al. in \cite{GKZ92}

Another reduction occurs in the case of $r=1$, 
T. Maeda \cite{maeda} and F. Holweck \cite{hol} studied dual Grassmannians i.e. duals of Pl\"{u}cker embeddings. Here, we focus only on this case since we want to give a complete and unified classification of components of singular locus. In addition, we also solve a linear algebra problem of independent interest which is about polynomial factors of determinants of a certain type of matrices in the section \ref{hessian}, see theorem \ref{determinant}\\

Now, we outline our results and introduce definitions and notations to be used throughout the paper.\\

Let $V$ be a vector space over complex field $\Cc$, $V^{\ast}$ be its dual. The set of one dimensional subspaces of $V$ is called the  projectivization of $V$ and denoted by $\Pp\left(V\right)$. For each point in $\Pp\left(V\right)$ we can associate a hyperplane in $V$. After regarding those hyperplanes as points, dual projective space $\Pp\left(V\right)^{\ast}\cong\Pp\left(V^{\ast}\right)$ is obtained.  
\begin{definition}
Let $X\subset\Pp^N=\Pp\left(\Cc^{N+1}\right)$ be a projective variety. The dual variety $X^{\vee}\subset\left(\Pp^N\right)^{\ast}$ is the closure  of the set of all tangent hyperplanes to $X$ i.e.
\begin{align}
X^{\vee}:=\overline{\left\{ F\,\vert\, p\in X_{sm}\,\,, \Pp T_pX\subseteq F\right\}}
\end{align}
where $\Pp T_pX$ is tangent space to $X$ at smooth point $p\in X_{sm}$.
\end{definition}

Now we focus on the case that $X$ is the Grassmannian variety $G(k,\Cc^N)$ which is the set of all $k$-dimensional vector subspaces of $\Cc^N$. Consider the Pl\"{u}cker embedding:

\begin{align}
G(k,\Cc^N)\rightarrow\Pp\left(\bigwedge^k\Cc^N\right)
\end{align}
which takes a $k$ dimensional subspace $L\subset\Cc^N$ to the one dimensional subspace $\bigwedge^k L\subset \bigwedge^k\Cc^N$. If we choose the coordinate matrix $K$ of size $k\times N$: \vspace{3pt}
\begin{align}\label{K}
K=\begin{bmatrix}
x^{1}_{1} & x^{1}_2  &\cdots &  x^{1}_{N} \\ 
x^2_1     &   x^2_2       &\cdots &  x^{2}_{N} \\
\vdots    &\vdots    &\ddots & \vdots    \\ 
x^k_1    &x^k_2    &\cdots & x^{k}_{N} 
 \end{bmatrix}
\end{align}
then, minors of this matrix give us local coordinates of ambient space, which are called Pl\"{u}cker coordinates. They are subject to quadratic relations which are called Pl\"{u}cker relations.\\

For a given index set $I=\left(i_1,\ldots,i_k\right)$, $\eta_I$ denotes the minor of coordinate matrix \ref{K} where columns are indexed by $I$. To describe the dual Grassmannian, we can use the multilinear form:
\begin{align}
\label{form}F\left(A,x\right)=\sum\limits_{1\leq i_1<i_2<\ldots<i_k\leq N} a_{i_1\ldots i_k}\eta_{i_1\ldots i_k}
\end{align} 
 
The form $F$ belongs to the dual Grassmannian $G(k,\Cc^N)^{\vee}$ if and only if $F$ and its partial derivatives with respect to local coordinates are zero at some nonzero point $x$, i.e. 
  the system of equations
\begin{gather}  
   F\left(A,x\right)=0,\hspace{1cm}\frac{\partial F\left(A,x\right)}{\partial x^{j}_{i}}=0
\end{gather}   
    for all i,j has a nontrivial solution for some point $x$.
\begin{definition} We call such a point $x$,   critical point of the form  $F\left(A,x\right)$\ref{form}.
\end{definition}
\begin{remark}\label{array}
Let $\cA$ be the space of multidimensional arrays i.e coefficients of the form $F=F\left(A,x\right)$ for $X=G\left(k,\Cc^N\right)$. Let $Z$ be the incidence variety:
\begin{align}\label{incidence1}
Z:=\left\{\left(A,x\right)\in \cA\times X\vert F\left(A,x\right)\in X^{\vee}\right\}
\end{align}
The image of the projection $\pi_1$ onto the first factor of $Z$ is the dual Grassmannian variety i.e. 
\begin{align}\label{projection}
\pi_1(Z)=X^{\vee}
\end{align}
\end{remark}

\begin{remark} There is a natural action of $G=GL\left(V\right)$ on $V$. There is an induced action of $G$ on $\bigwedge^kV$ and $\cA$. Explicitly:
\begin{align}
F\left(A,g.\left(v_1\wedge\ldots\wedge v_k\right)\right)=F\left(A,gv_1\wedge\ldots\wedge gv_k \right)=F\left(A\cdot g,v_1\wedge\ldots\wedge v_k\right)
\end{align}
for all $g\in G$. The coefficients of the new form $F\left(A\cdot g,x\right)$ are given by
\begin{align}
\left(A\cdot g\right)_{j_1\ldots j_k}=\sum\limits_{i_1<i_2<\ldots<i_k} a_{i_1i_2\ldots i_k}.g^{j_1\ldots j_k}_{i_1\ldots i_k}
\end{align}
where $g^{j_1\ldots j_k}_{i_1\ldots i_k}$ is the $k\times k$ minor of $g$ with column indices $j_1,\ldots,j_k$ and row indices $i_1,\ldots,i_k$.
\end{remark}

To analyze singularities, an important tool is the Hessian matrix associated to form \ref{form}  which is the matrix of the second partial derivatives in the local coordinates:
\begin{align}\label{h1}
H(F)=H\left(F(A,x)\right)=\vert\vert \cfrac{\partial^2F\left(A,x\right)}{\partial x^i_j\partial x^{i'}_{j'}}\vert\vert
\end{align}
We analyze its algebraic properties in section \ref{hessian}.

\begin{definition} Cusp type locus is the subvariety of $X^{\vee}$ such that determinant of Hessian matrix vanishes. Formally:
\begin{align*}
X^{\vee}_{cusp}:=\left\{F\mid \exists p\in X \;\text{s.t}\; \Pp T_pX\subset F \;\text{and}\;\det H\left(F\right)\vert_p=0\right\}.
\end{align*}
\end{definition}

We prove the following in section \ref{cuspsection}:
\begin{theorem}
For $k\geq 3$ and $N\geq 9$ or $k\geq 4$ and $N\geq 2k$, $X^{\vee}_{cusp}$ is an irreducible hypersurface in $X^{\vee}$. 
\end{theorem}

\begin{definition}\label{definitionofnode}
Node type singular locus is the subvariety of $X^{\vee}$ which is the set of forms tangent to $X$ at least two distinct points $p,q\in X$. Formally:
\begin{align*}
X^{\vee}_{node}:=\overline{\left\{F\mid \exists p,q\in X \quad\text{such that}\quad  \Pp T_pX,\Pp T_qX\subset F \right\}}.
\end{align*}
\end{definition}
We will show the following in section \ref{genericnodesection}:
\begin{theorem}\label{resultnode}
For $k\geq 3$ and $N\geq 9$ or $k\geq 4$ and $N\geq 2k$, $X^{\vee}_{node}$ is an irreducible hypersurface in $X^{\vee}$. 
\end{theorem}

It is well known that under the $SL_N$ action the space $\bigwedge^3\Cc^N$ for $N=6,7,8$ has finitely many orbits \cite{KW1}, \cite{KW2}. In particular they are prehomogeneous vector spaces, see \cite{kimura} for the theory. For those, we show that $X^{\vee}_{sing}=X^{\vee}_{cusp}$ by simply checking the orbit representatives in the section \ref{decompositionofsingularlocus}. For the other cases we show that $X^{\vee}_{cusp}\neq X^{\vee}_{node}$ in theorem \ref{cuspnodedifferent}.


Now we present the main theorem by putting together all the above results:
\begin{theorem*}[MAIN]
The singular locus of the dual Grassmannian $G(k,\Cc^N)^{\vee}$  is of codimension two in $\bigwedge^k\Cc^N$ for $N\geq 2k$, $k\geq 4$ or $N\geq 9$ and $k=3$. For these cases singular locus $X^{\vee}_{sing}$ has two components $X^{\vee}_{node}$ and $X^{\vee}_{cusp}$ which are irreducible, both having codimension one in $X^{\vee}$.
For the cases $k=3$ and $N=6,7,8$, the codimensions of $X^{\vee}_{cusp}$ in $X^{\vee}$ are $4,3,2$ respectively and $X^{\vee}_{node}$ is a subvariety of $X^{\vee}_{cusp}$. 
\end{theorem*}

 Here we list a brief account of some earlier works in the literature about the dual Grassmannian varieties including some versions of the main theorem. In this work we focus on the case $k\geq 3$, since if $k=2$ the dual variety is degenerate skew symmetric matrices, i.e. vanishing locus of pfaffians which is a well known case. Here the space of arrays $\cA$ [\ref{array}] is simply skewsymmetric matrices. It was shown in \cite{knop} that the dual Grassmannian is hypersurface for all $k,N$ except $k=2$ and $N$ is odd. We reproduce that observation in section \ref{hessian} by using a criteria due to Katz \cite{katz} and its generalization in \cite{WZ2} about Hessian matrices. It is enough to work with $N\geq 2k$ because of the isomorphism $\bigwedge^k\Cc^N\cong\bigwedge^{N-k}(\Cc^{N})^{\ast}$. The degree of the defining equations was computed by A. Lascoux \cite{lasc}. The result: the dual Grassmannian $G(k,\Cc^N)^{\vee}$, $k\geq 3$ is normal if and only if $k=3$ and $N=6,7,8$ was proven in \cite{hol} and the method is based on the calculation of the dimension of the secant variety of $G(k,\Cc^N)^{\vee}$. We found an earlier work \cite{maeda} then \cite{hol} in which similar results are obtained without specifying dimensions.

Organization of the paper is: in the first section we study the cusp component. Sections \ref{genericnodesection} and \ref{specnode} are for the node component. Their relation is given in the section \ref{decompositionofsingularlocus}. The last section is devoted to determinants of Hessian matrices. The style is very close to "singularities of hyperdeterminants \cite{WZ}". We do not hesitate to use the notation developed in it and to modify their results whenever it is possible.
\subsection{Acknowledgments} 
The results of this paper were obtained during my Ph.D. studies at Northeastern University and partially supported from the funding of Andrei Zelevinsky, NSF 1103813. I am indebted to my supervisors Kiyoshi Igusa, Gordana Todorov and Jerzy Weyman for their support,  endless patience and encouragement for this work. I appreciate the helpful comments of Anthony Iarrobino and Venkatramani Lakshmibai for the earlier versions of this manuscript. I would like to thank Alexander Klyachko for introducing me to this field.

\section{Cusp Type Singularities}\label{cuspsection}

In this section we study the cusp component. It is well known that the dimension of a variety at a smooth point is the dimension of the tangent space, hence we will compare the dimensions of the tangent spaces of $X^{\vee}$ and $X^{\vee}_{cusp}$ at a point $A\in\cA$.

Let $Y$ be the span of $k$ linearly independent vectors, i.e. $\langle v_1,\ldots,v_k\rangle\in Y$. Consider the subset of $Y$
\begin{align}\label{affineE}
E=\left\{\langle v_1,\ldots,v_k\rangle\in Y\vert v_i=e_i+\sum\limits_{k<j} x^i_je_j\right\}
\end{align} where $\left\{e_i\right\}_{1\leq i\leq N}$ is the standard basis of $\Cc^N$. Notice that the image of the set $E$ under the Pl\"{u}cker map gives the chart in $\Pp\left(\bigwedge^k\Cc^N\right)$ around the point $[1:0\ldots:0]$. When all coordinates are zero, we denote the point by $x^0$ i.e. $x^0=\langle e_1\ldots e_k\rangle$. It will be convenient for us to dehomogenize the multilinear form $F\left(A,x\right)$ corresponding to $A\in\cA$ by setting all $x^{j}_i=\delta_{i,j}$ in \ref{K} where $i,j\leq k$. More precisely, we restrict the form $F$ onto $E$ i.e. $F(A,v_1\wedge\ldots\wedge v_k)$, where $v_1\wedge\ldots\wedge v_k\in E$ and it becomes:

\begin{align}
a_{12\ldots k}+\sum\limits_{k< j} a_{12\ldots\hat{j}\ldots k}x^p_j+ \sum\limits_{k<j<j'} a_{12\ldots\hat{j}\ldots\hat{j'}\ldots k}\left(x^p_jx^{p'}_{j'}-x^p_{j'}x^{p'}_{j}\right)+\ldots
\end{align}
where $ a_{12\ldots\hat{j}\ldots k}$ and $a_{12\ldots\hat{j}\ldots\hat{j'}\ldots k}$ are the coefficients which we replace the positions of $p$ and $p'$ by $j,j'$ respectively. Since the notation is not easy to follow, we introduce the following:

\begin{notation}\label{not1}
We use: 
\begin{itemize}
\item $a:=a_{12\ldots k}$
\item $a^t_p:=a_{12\ldots\hat{t}\ldots k}$ where $t$ is the value of the position $p$.
\item In general: $a^{t_1t_2\ldots t_r}_{p_1p_2\ldots p_r}:= a_{\hat{t_1}\ldots\hat{t_2}\ldots\hat{t_r}\ldots }$ where each $p_i$, $1\leq p_i\leq k$ denotes the position of the index $t_i$, $k+1\leq t_i\leq N$.
\item For a given unordered index set $\left(i_1,i_2,\ldots,i_r\right)$, let $\sigma\left(i_1,i_2,\ldots,i_r\right)=(-1)^{\tau\left(i_1,i_2,\ldots,i_r\right)}$, where $\tau$ gives the number of transpositions to convert  $\left(i_1,i_2,\ldots,i_r\right)$ into ordered tuple. Therefore $a^{t_1t_2\ldots t_r}_{p_1p_2\ldots p_r}=\sigma(\hat{t_1}\ldots\hat{t_2}\ldots\hat{t_r})a^{t^{'}_1 t^{'}_2\ldots t^{'}_r}_{p_1p_2\ldots p_r} $ with $t^{'}_1<\ldots <t^{'}_r$ and $\left\{t_1,\ldots,t_r\right\}=\left\{t^{'}_1,\ldots,t^{'}_r\right\}$

\end{itemize}
\end{notation}

For example we will write the coefficient $a_{1,2,\ldots,t,\ldots,t',\ldots,k}$ of the quadratic term as $a^{tt'}_{pp'}$ where the positions of $t$ and $t'$ in $\left(1,\ldots,t,\ldots,t'\ldots,k\right)$ are $p$ and $p'$ respectively. By abuse of notation we do not add "$-$" signs, it is clear that $a^{tt'}_{pp'}=-a^{t't}_{pp'}$. Indeed $\cA$ [\ref{array}] is the space of skew symmetric arrays.

Under the new notation, the form $F$ restricted to $E$ can be written as:
\begin{align}\label{dehomform}
F\left(A,x\right)=a+\sum\limits_{p,t}a^{t}_{p}x^{p}_{t}+\sum\limits_{t,t^{\prime},p,p^{\prime}}a^{tt^{\prime}}_{pp^{\prime}}x^{p}_{t}x^{p^{\prime}}_{t^{\prime}}+\ldots
\end{align}

\begin{example}\label{example} Let $X^{\vee}=G(3,\Cc^6)^{\vee}$. Form $F$ on $E$:
\begin{gather*}
a+\left[a^4_3x^3_4+a^5_3x^3_5+a^6_3x^3_6-a^4_2x^2_4-a^5_2x^2_5-a^6_2x^2_6+a^4_1x^1_4+a^5_1x^1_5+a^6_1x^1_6\right] \\ 
a^{45}_{23}(x^2_4x^3_5-x^3_4x^2_5)+a^{46}_{23}(x^2_4x^3_6-x^2_6x^3_4)+a^{56}_{23}(x^2_5x^3_6-x^2_6x^3_5)\\-a^{45}_{13}(x^1_4x^3_6-x^1_6x^3_4)-a^{46}_{13}(x^1_4x^3_6-x^1_6x^3_4)-a^{56}_{13}(x^1_5x^2_6-x^1_6x^2_5)\\+a^{45}_{12}(x^1_4x^2_5-x^1_5x^2_4)+a^{46}_{12}(x^1_4x^2_6-x^1_6x^2_4)+a^{56}_{12}(x^1_5x^2_6-x^1_6x^2_5)\\+a^{456}_{123}(x^1_4x^2_5x^3_6-x^1_4x^2_6x^3_5+x^1_5x^2_6x^3_4-x^1_5x^2_4x^3_6+x^1_6x^2_4x^3_5-x^1_6x^2_5x^3_4)
\end{gather*}
\end{example}

Let $\nabla\left(x\right)$ be the set of arrays $A$ in $\cA$ having $x$ as a critical point i.e.
\begin{align}\label{nablax}
\nabla\left(x\right)=\left\{A\vert F\left(A,x\right)\in X^{\vee}\right\}\end{align}
In particular: 
\begin{align*}
\nabla\left(x^0\right)&=\left\{A\vert F\left(A,x^0\right)\in X^{\vee}\right\}\\
&=\left\{A\vert F\left(A,x^0\right)=\cfrac{\partial F(A,x^0)}{\partial x^i_j}=0,\,\,\forall i,j\, \right\}\\
&=\left\{A\vert a=0,\,\,a^t_p=0,\,\,\forall t,p\right\}
\end{align*}
Consider the Hessian matrix of the dehomogenized form $F$ \ref{dehomform}  evaluated at $x^0$:
\begin{center}
$H(F(A,x))\vert_{x^0}=\vert\vert \cfrac{\partial^2 F}{\partial x^i_j\partial x^{i'}_{j'}}\vert\vert_{x^0}$
\end{center}
The nonzero entries of the Hessian matrix \ref{h1} evaluated at $x^0$ are the coefficients $a^{tt'}_{pp'}$ of the form $F$ \ref{dehomform}.
\begin{example} The Hessian matrix of the form $F$ in the previous example \ref{example} is the following $9\times 9$ symmetric matrix with skew symmetric blocks:\\

  \begin{center}
  $\left[ \begin {array}{ccc|ccc|ccc} 0&0&0&0&a_{{12}}^{45}&a_{{12}}^{46}&0&-a_{{13}}^{45}&-a_{{13}}^{46}\\ 0&0&0&-a_{
{12}}^{45}&0&a_{{12}}^{56}&a_{{13}}^{45}&0&-a_{{13}}^{56}
\\0&0&0&-a_{{12}}^{46}&-a_{{12}}^{56}&0&a_{{13}
}^{46}&a_{{13}}^{56}&0\\\hline 
0&-a_{{12}}^{45}&-a_{{
12}}^{46}&0&0&0&0&a_{{23}}^{45}&a_{{23}}^{46}\\
a_{{12}}^{45}&0&-a_{{12}}^{56}&0&0&0&-a_{{23}}^{45}&0&a_{{23}}
^{56}\\ a_{{12}}^{46}&a_{{12}}^{56}&0&0&0&0&-a_{
{23}}^{46}&-a_{{23}}^{56}&0\\ \hline 0&a_{{13}}^{45}&a
_{{13}}^{46}&0&-a_{{23}}^{45}&-a_{{23}}^{46}&0&0&0
\\-a_{{13}}^{45}&0&a_{{13}}^{56}&a_{{23}}^{45}
&0&-a_{{23}}^{56}&0&0&0\\-a_{{13}}^{46}&-a_{{13
}}^{56}&0&a_{{23}}^{46}&a_{{23}}^{56}&0&0&0&0\end {array}
 \right] 
$
  \end{center}
  
\end{example}

   In this set up the cusp variety can be expressed as:
\begin{align}
\nabla^{0}_{cusp}=\left\{A\in\nabla\left(x^0\right)\vert \det H\left(F(A,x)\vert_{x^0}\right)=0\right\}
\end{align}
and by the action of group $G$: $X^{\vee}_{cusp}=\nabla^{0}_{cusp}\cdot G$. 
\begin{proposition}
The cusp variety $X^{\vee}_{cusp}$ is irreducible.
\end{proposition}
\begin{proof}
It is enough to show that $\nabla^0_{cusp}$ is irreducible. This follows from the study of the polynomial factors of the determinant of the Hessian matrix. Proofs are given at the last section. Here, we just mention that, the determinant is irreducible for all $k=3, N\geq 8$ or $k\geq 4, N\geq 2k$. For $k=3, N=6,7$, it is the cube and the square of an irreducible polynomial respectively which implies irreducibility of the variety.
\end{proof}

Instead of working with the group $G=GL\left(\Cc^{N}\right)$, we use $E\subset G$ where $x\in E$ has the following structure:
\begin{align}
 \begin{bmatrix}
I_{k\times k}& 0 \\ 
M& I_{\left(N-k\right)\times\left(N-k\right)} \\
 \end{bmatrix}
\end{align}
where the first $k$ columns are transpose of the vectors given  in $E$ i.e. 
\begin{center}
$M=\begin{bmatrix}
x^{1}_{k+1} & x^2_{k+1}   &\cdots & x^k_{k+1}  \\ 
x^{1}_{k+2}    &   x^2_{k+2}       &\cdots &  x^{k}_{k+2} \\
\vdots    &\vdots    &\ddots & \vdots    \\ 
x^1_{N}& x^2_N    &\cdots & x^{k}_{N} 
 \end{bmatrix}$
\end{center}

$E$ acts transitively by translations on the affine space $E$ \ref{affineE}. The action on the coefficients i.e. $\cA$ is:
\begin{align}
\left(A\cdot x\right)^{t_1\ldots t_r}_{p_1\ldots p_r}=\cfrac{\partial^r F(A,x)}{\partial x^{p_1}_{t_1}\cdots \partial x^{p_r}_{t_r}}
\end{align}

Notice that for $g\in G$, $\nabla^0_{cusp}\cdot g$ only depends on the point $g^{-1}x^0$ since:
\begin{align}
A\cdot g\in\nabla(x^0)&\iff F(A\cdot g,x^0)\in X^{\vee}\\
&\iff F(A,g(g^{-1}x^0))\in X^{\vee}
\end{align}
Therefore $X^{\vee}_{cusp}=\overline{\nabla^0_{cusp}\cdot E}$.

\begin{theorem}\label{thmcusp}
The tangent space $T_A X^{\vee}_{cusp}$ of the cusp variety at the generic point $A$ of $\nabla^0_{cusp} $ has codimension $2$ in $T_A\cA$ i.e the tangent space of all arrays. 
\end{theorem}

\begin{proof} The first part of the proof work for any $k,N$. Consider the incidence variety:
\begin{center}\label{inc2}
 $\tilde{Z}=\left\{\left(A,x\right)\in\cA\times E\vert A\cdot x\in\nabla\left(x^0\right)\right\}$
 \end{center}
By the definition of $\nabla(x^0)$ and the group action $E$ on $\cA$, $T_{(A,x^0)}\tilde{Z}$ is given by the equations:
\begin{align}\label{eq1*}
da=0,\,\, da^t_p+\sum\limits_{t^{'},p^{'}} a^{tt'}_{pp'}dx^{p'}_{t'}=0
\end{align}
where the differentials $da,da^t_p,\ldots$ and $dx^p_t$ are the usual coordinates in $T_A\cA$ and $T_{x^0}E$ respectively. This is because of :
\begin{center}
$A\cdot x\in\nabla(x^0)\iff F(A,X)\vert_{x^0}=0,\,\cfrac{\partial}{\partial x^i_j}F(A,x)\vert_{x^0}=0  $
\end{center}
so the differentials $dF(A,X)\vert_{x^0}=0,\,d\left(\cfrac{\partial}{\partial x^i_j}F(A,x)\right)\vert_{x^0}=0 $ give the equations \ref{eq1*}.\\
Since the vanishing locus of determinants of Hessian matrices is the cusp variety, we obtain it as a subvariety of $\tilde{Z}$ i.e.
\begin{align}\label{inccusp}
\tilde{Z}_{cusp}&=\left\{\left(A,x\right)\in\cA\times E\vert A\cdot x\in\nabla^0_{cusp}\right\}\\
&=\left\{\left(A,x\right)\in\tilde{Z}\vert \det H(F(A,x))\vert_{x^0}=0\right\}
\end{align}
Hence we obtain one more equation for the tangent space  $T_{(A,x^0)}\tilde{Z}$ which is the differential of the determinant i.e. $d\left(\det H(F(A,x))\right)\vert_{x^0}=0$
After expanding it, we get:

\begin{align}\label{eq2}
\sum\limits_{t,t',p,p'}\hat{C}^{pp'}_{tt'}\left(da^{tt'}_{pp'}+\sum\limits_{t^{''},p^{''}}a^{tt't^{''}}_{pp'p^{\prime\prime}}dx^{p^{\prime\prime}}_{t^{''}}\right)=0
\end{align}

where $\hat{C}^{pp'}_{tt'}$ is the cofactor of the Hessian matrix corresponding the entry $a^{tt'}_{pp'}$. By the theorem \ref{hessian}, it is irreducible except $(k,N)=(3,6),(3,7)$.

It is clear that the projection onto the first factor in \ref{inccusp} gives the tangent space at the cusp variety i.e.
\begin{center}
$\pi_1\left(T_{(A,x^0)}\tilde{Z}_{cusp}\right)=T_AX^{\vee}_{cusp}$
\end{center}
As a result, the tangent space of $X^{\vee}_{cusp}$ at the point $A\in\nabla^0_{cusp}$ is given by the linear equations between coordinates $da,da^t_p,\ldots$, that are the consequences of \ref{eq1*}, \ref{eq2}. One of them is $da=0$.
We can choose $A\in\cA$ such that the Hessian matrix $H(F(A,x))$ of the form associated to $A$ is of corank one as described in the subsection \ref{ranklemmas}. Therefore the system of equations in \ref{eq1*} is equivalent to matrix equation:
\begin{center}
$H(F(A,x))\vert_{x^0}\left[dx^p_t\right]=-\left[da^t_p\right]$
\end{center}
Now, we assume that the variety is not exceptional i.e. $(k,N)\notin\left\{(3,6),(3,7),(3,8)\right\}$. By the construction in \ref{ranklemmas}, there is  Hessian matrix  $H(F(A,x))\vert_{x^0}$ of corank one, its adjoint matrix $Adj(H(F(A,x))\vert_{x^0})$ is of rank one, and their product is zero matrix obviously. Therefore:
\begin{align}
Adj(H(F(A,x))\vert_{x^0})\left[da^t_p\right]=0
\end{align}
which gives one relation between coordinates of $da^t_p$.\\

Now, we show that equation \ref{eq2} cannot produce additional relations between the coordinates $da, da^t_p,da^{tt'}_{pp'},\ldots$ etc. Assume to the contrary that \ref{eq2} is in the linear span of \ref{eq1*} i.e. 
\begin{center}
$d\left(\det H(F(A,x))\right)\vert_{x^0}+\sum\limits_{p,t}c^p_t\left(da^t_p+\sum\limits_{t^{'},p^{'}} a^{tt'}_{pp'}dx^{p'}_{t'}\right)=0$
\end{center}
for some coefficients $c^p_t$. Since we are looking relations between the coordinates $da, da^t_p,\ldots$ not involving $dx^i_j$, the coefficient of $dx^{p''}_{t''}$ 
\begin{align}
\sum\limits_{t,t',p,p'}\hat{C}^{pp'}_{tt'}a^{tt't^{''}}_{pp'p^{\prime\prime}}+\sum\limits_{t,p}c^p_t\sum\limits_{t',p'}a^{t't''}_{p'p''}
\end{align}
has to be zero for all $t'',p''$. We multiply it by coefficients $c^{p''}_{t''}$ and take sum over all indexes to get:
\begin{align}
\sum\limits_{t'',p''}\sum\limits_{t,t',p,p'}\left(\hat{C}^{pp'}_{tt'}a^{tt't^{''}}_{pp'p^{\prime\prime}}+c^p_ta^{t't''}_{p'p''}\right)c^{p''}_{t''}=\\
\sum\limits_{t,t',p,p',t'',p''}\hat{C}^{pp'}_{tt'}a^{tt't^{''}}_{pp'p^{\prime\prime}}c^{p''}_{t''}+\sum\limits_{t,t',p,p'} \sum\limits_{t'',p''}\left(a^{t't''}_{p'p''}c^{p''}_{t''}\right)
\end{align}
Notice that we do not have any condition on the coefficients $a^{tt't^{''}}_{pp'p^{\prime\prime}}$, we can choose them so that the first summand is nonzero. However, the second summand is the determinant of the Hessian matrix which is zero. This creates contradiction.
\end{proof}

\section{The Generic Node Component}\label{genericnodesection}
In this section we study the generic node component. Recall that in \ref{definitionofnode} we defined $X^{\vee}_{node}=\overline{\left\{F\mid \exists p,q\in X \quad\text{such that}\quad  \Pp T_pX,\Pp T_qX\subset F \right\}}.$ Another formulation is due to the incidence variety:
\begin{align}
Z^2:=\left\{\left(A,x,y\right)\in \cA\times X\times X\vert \left(A,x\right)\in Z,\,\left(A,y\right)\in Z,\,\text{and}\,\,x\neq y\right\}
\end{align}
Let $\pi_1$ be the projection of $Z^2$ onto the first coordinate i.e. $\pi_1(A,x,y)=A$, then $\nabla^0_{node}=\pi_1\left(Z^2\right)$. The closure $\overline{\nabla^0_{node}}$ is the node variety $X^{\vee}_{node}$. If $x$ and $y$ contains common vectors i.e. $x=\langle v_1,\ldots,v_k\rangle$, $y=\langle w_1\ldots,w_k\rangle$ and $\left\{v_1,\ldots,v_k\right\}\cap\left\{w_1,\ldots,w_k\right\}$ is non empty, we call the node variety 'special' and study them in the next section \ref{specnode}. In this chapter we assume that intersection empty and call it 'generic node component'. 
Recall that $\nabla\left(x^0\right)=\left\{A\vert F(A,x^0)\in X^{\vee}\right\}$ \ref{nablax}. Similar to the definition of $x^0$, let $x'$ be the point $e_{N-k+1}\wedge\ldots\wedge e_N$, so it is the point $[0:\ldots:0:1]$ in $\Pp\left(\bigwedge^k\Cc^N\right)$. Now we define the forms with critical point $x'$ i.e. $\nabla\left(x'\right)=\left\{A\vert F(A,x')\in X^{\vee}\right\}$. By using the group action of $G$ on $X\times X$ we consider:
\begin{align}
\nabla_{node}(\emptyset):=\overline{\left(\nabla(x^0)\cap\nabla(x')\right)\cdot G}
\end{align} 
i.e. forms such that they are tangent to $X$ at $x^0 $ and $x'$. 
\begin{remark} $\nabla_{node}(\emptyset)$ is the variety studied in \cite{hol}. 
\end{remark}
Another way to restate the definition of the generic node component is the following. Let $I^f$ and $I^{\ell}$ be the ordered sets $\left\{1,2,\ldots,k\right\}$ and $\left\{N-k+1,\ldots,N\right\}$ respectively. Recall that we chose $N\geq 2k$, which implies $I^f\cap I^{\ell}=\emptyset$. 
\begin{definition}\label{twosets} The generic node variety is the set of forms which are tangent to $X$ at $I^f$ and $I^{\ell}$ i.e:
\begin{align*}
\left\{\sum a_{i_1\ldots i_k}\eta_{i_1\ldots i_k}\mid a_{i_1\ldots i_k}=0, \text{when} \,\vert I^f\cap 
 \left\{i_1\ldots i_k\right\}\vert \geq k-1\, \text{or}\,\vert I^{\ell}\cap \left\{i_1\ldots i_k\right\} 
 \vert \geq k-1 \right\} 
\end{align*}

 for all possible indices $\left\{i_1\ldots i_k\right\}$. 
\end{definition}

As we did for the cusp component, we compute the dimension of the tangent space at a generic point of the intersection $\nabla(x^0)\cap\nabla(x')$.

\begin{theorem} \label{thmnode}
For all $k\geq 3$ and $N\geq 2k$, except $k=3, N=6,7,8$,  the tangent space $T_A\nabla_{node}(\emptyset)$ at the generic $A\in \nabla(x^0)\cap\nabla(x')$ is the subspace of codimension two in $T_A\cA$ given by equations 
\begin{align}
da_{12\ldots k}=0,\hspace{20pt} da_{N-k+1\ldots N}=0
\end{align}
For the exceptional cases:
\begin{align}
 \nabla_{node}(\emptyset)\subset X^{\vee}_{cusp}.
 \end{align}
\end{theorem}


\begin{proof}
Recall that we used the affine chart $E$ around the point $[1:0\ldots :0]$. The matrix:
\begin{align}
w= \begin{bmatrix}
0& I_{N-k\times N-k} \\ 
I_{k\times k}&0 \\
 \end{bmatrix}
\end{align}

sends $E$ to $wE$ where $wE$ is the chart around the point $[0:\ldots:0:1]$. Similar to the proof of theorem \ref{thmcusp}, we consider 
\begin{align}
\tilde{Z^2}:=\left\{\left(A,x,y\right)\in\cA\times E\times wE\vert (A,x,y)\in Z^2\right\}
\end{align}
The projection onto the first factor is the generic node component i.e.
\begin{center}
$\nabla_{node}(\emptyset)=\overline{\pi_1(\tilde{Z^2})}$
\end{center}.

Therefore, for a generic $A\in \nabla(x^0)\cap\nabla(x') $, the tangent space $T_A\nabla_{node}(\emptyset)=\pi_1\left(T_{(A,x^0,x')}\tilde{Z^2}\right)$. Here we introduce the dual notation similar to \ref{not1}:
\begin{notation} To simplify the form which vanishes around $x'$ we have:
\begin{itemize}
\item $^wa:=a_{N-k+1\ldots N}$
\item $^wa^t_p:=a_{N-k+1\ldots \hat{t}\ldots N}$ where $1\leq t\leq N-k$ is the value at the position $p$, $N-k+1\leq p\leq N$
\end{itemize}
\end{notation}

The tangent space $T_{\left(A,x^0,x'\right)}\tilde{Z}^2$ is given by equations  
\begin{align}\label{eq3}
da=0, da^t_p+\sum\limits_{t^{'},p^{'}} a^{tt'}_{pp'}dx^{p'}_{t'}=0
\end{align}
together with 
\begin{align}\label{eq4}
d{}^wa=0, d{}^wa^t_p+\sum\limits_{t^{'},p^{'}} {}^wa^{tt'}_{pp'}dx^{p'}_{t'}=0
\end{align}

The tangent space at a generic $A\in \nabla\left(x^0\right)\cap\nabla\left(x'\right)$ is given by $da=d{}^wa=0$ if and only if the other equations in \ref{eq3} and \ref{eq4} do not produce additional relations between $da^t_p, d{}^wa^t_p$. This is equivalent to existence of invertible Hessian matrices, since the system of equations in \ref{eq3}, \ref{eq4} can be collected into:
\begin{gather}
H(F(A,x))\vert_{x^0}[dx^p_t]=-[da^t_p]\\
H(F(A,x))\vert_{x'}[d^wx^p_t]=-[d^wa^t_p]
\end{gather}

If both of the Hessian matrices $H(x^0):=H(F(A,x))\vert_{x^0}, H(x'):=H(F(A,x))\vert_{x'}$ are invertible, then it holds. \\

Assume that $k\geq 5$. There is no common entries of $H(x^0),H(x')$ and those entries are arbitrary, so their determinants are nonzero. Therefore the generic node component is of codimension one in $X^{\vee}$.

If $k=4$, there are $36$ nonzero common entries of $H(x^0)$ and $H(x')$. Specifically the common entries of $H(x^0)$ and $H(x')$ are of the form $a^{tt'}_{pp'}$ where $t,t'\in I^{\ell}$ and $p,p'\in I^f$.  After row and column operations, they satisfy
\begin{align}\label{k4hessiantuple}
\left(H(x^0)_{{\alpha\beta}}\right)_{\gamma\theta}=\left(H(x')_{{\overline{\gamma\theta}}}\right)_{\overline{\alpha\beta}}
\end{align}

where inner subscripts give position of skew symmetric block in the whole Hessian, outher subscripts are inner position in the block, overline means complementary indices i.e. $\overline{\alpha\beta}=\left\{1,2,3,4\right\}-\left\{\alpha,\beta\right\}$. For $N=8$, there exists such matrix \ref{48invertible}. By using the specialization methods \ref{hessian}, we can use \ref{48invertible} to obtain larger invertible Hessian matrices for $k=4$ satisfying \ref{k4hessiantuple}.

The most complicated case occurs when $k=3$, because there are more constraints on the entries of $H(x^0)$ and $H(x')$. We need to find $H(x^0)$ and $H(x')$
\begin{align}\label{tuple1}
H(x^0)=\left[ \begin {array}{ccc} 0&A_{{12}}&A_{{13}}\\\noalign{\medskip}-A_{{12}}&0&A_{{23}}\\\noalign{\medskip}-A_{{13}}&-A_{{23}}&0
\end {array} \right], \vspace{10pt} H(x')=\left[ \begin {array}{ccc} 0&B_{{12}}&B_{{13}}\\\noalign{\medskip}-B_{{12}}&0&B_{{23}}\\\noalign{\medskip}-B_{{13}}&-B_{{23}}&0
\end {array} \right] 
\end{align}

satisfying:

\begin{enumerate}[label=\roman*)]
\item \label{c1} For each $A_{ij}$ (respectively, $B_{ij}$) the last (respectively, the first) $3\times 3 $ block is zero matrix.
\item \label{c2} The first row of $B_{12}$ (respectively, $B_{13},B_{23}$) is the transpose of $(N-3)$th (respectively, $(N-4)$th, $(N-5)$th) column of $A_{23}$ 
\item \label{c3} The second row of $B_{12}$ (respectively,$B_{13},B_{23}$) is the transpose of $(N-3)$th (respectively, $(N-4)$th, $(N-5)$th) column of $A_{13}$ 
\item \label{c4} The third row of $B_{12}$ (respectively, $B_{13},B_{23}$) is the transpose of $(N-3)$th (respectively, $(N-4)$th, $(N-5)$th) column of $A_{23}$ 
\end{enumerate}
In section \ref{hessian} we list matrices satisfying the above conditions for $N=9,10,11$. Then by using the specialization methods \ref{ranklemmas}, we prove that for any $N\geq 9$ there are pair of matrices satisfying all \ref{c1},\ref{c2}\ref{c3}\ref{c4}.\\

We continue our analysis with the remaining cases i.e. $k=3, N=6,7,8$. For $N=6$, there is no nonzero entry satisfying conditions, hence there is no generic node component. For $N=7,8$, it is nonzero, however determinant of Hessian matrices vanish, so they are subvarieties of $X^{\vee}_{cusp}$. 

\end{proof}

\section{Special Node Components}\label{specnode}

In the analysis of the generic node component, we used the sets $I^{\ell}$, $I^f$ \ref{twosets} and we imposed the condition: $I^f\cap I^{\ell}=\emptyset$. Now we analyze the case that we allow nontrivial intersection. 

 Let $J$ be an ordered set such that 
\begin{align}
J\subset I^f\cup I^{\ell}, \quad \vert J\vert=k
\end{align} 
The complement of $J$ in $I^f\cup I^{\ell}$ is denoted by $\overline{J}$, i.e. $J\cap \overline{J}=\emptyset$ and $J\cup \overline{J}=I^f \cup I^{\ell}$.\\
We use the following modification of the construction in \cite{WZ} to the dual Grassmannian:
Let $x\left(J\right)\in Y$ be the point with coordinate matrix $K$ of size $k\times N$ such that $K_{i,J(i)}=1$ otherwise zero where $J(i)$ means $i$th element of the ordered set $J$. In particular if $J=I^f$, it gives the point $[1:0:\ldots :0]$ in Pl\"ucker embedding.
Recall that $\nabla\left(x\right)$ is the set of arrays $a\in\cA$ such that $x$ is the critical point of $A$ \ref{nablax}. $\nabla\left(x\right)$ is a vector subspace in $\cA$ and it has codimension $k\left(N-k\right)+1$. We use star of a multiindex $(i_1,\ldots,i_k)$ which is the set of all multiindices that differs from $(i_1,\ldots,i_k)$ in at most one position. 
Let $P\subset I^f$. Replacement $r(P)$ is the multiindex in which we replace the $\theta$th position of $I^f$ with the $\theta$th position of $J$. 
We have :
\begin{align}
\nabla\left(x\left(J\right)\right)=\left\{A\in\cA\vert a_{i_1\ldots i_k}=0,\quad  \,\forall (i_1\ldots i_k)\,\text{in the star of}\, J\right\}
\end{align}
By using this we get:
\begin{align}
\nabla_{node}\left(J\right)=\overline{\left(\nabla\left(x\left(I^f\right)\right)\cap\nabla\left(x\left(J\right)\right)\right)\cdot G}
\end{align}
where $J\neq I^{\ell}$, and bar denotes Zariski closure. 
Indeed the node component $X^{\vee}_{node}$ described in  \ref{definitionofnode} has the decomposition according to sets $J$, i.e. $X^{\vee}_{node}=\bigcup_{J}\nabla_{node}(J)$.
We studied the case $J=I^{\ell}$ in the previous section \ref{genericnodesection}. 
Observe that the star of multiindices $I^f$ and $I^{\ell}$ do not meet each other. Therefore, $\nabla\left(x\left(I^f\right)\right)\cap \nabla\left(x\left(I^{\ell}\right)\right)$ has codimension $2\left(1+k\left(N-k\right)\right)$ in $\cA$ as a vector subspace. Using the action of the group $G$, we see that $\nabla(x)\cap\nabla(y)$ has codimension  $2\left(1+k\left(N-k\right)\right)$ whenever $\nabla(x)\cap\nabla(y)\subset\nabla_{node}(\emptyset)$. By the same argument we conclude that $\nabla\left(x\left(I^f\right)\right)\cap \nabla\left(x\left(J\right)\right)$ has codimension  $2\left(1+k\left(N-k\right)\right)$, whenever $\vert I^f\cap J\vert \leq k-3$.\\
If  $\vert I^f\cap J\vert \leq k-2$, and $J-I^f=\left\{\alpha,\alpha'\right\}$, $I^f-J=\left\{t,t'\right\}$, star of multiindices has $4$ intersections. They are:
\begin{align}
a^{\alpha}_{t},\,a^{\alpha}_{t'}\,,\,a^{\alpha'}_{t},\,a^{\alpha'}_{t'}
\end{align}
Hence in this case codimension is  $2\left(k\left(N-k\right)-1\right)$ in $\cA$. \\
If  $\vert I^f\cap J\vert \leq k-1$, $\nabla_{node}(J)$ is not a component of $X^{\vee}_{node}$ by the next proposition.

\begin{proposition}\label{nodesubvarietycusp} For $G\left(k,\Cc^N\right)^{\vee}$, if $\vert J\cap I^f\vert=k-1$ then $\nabla_{node}\left(J\right)\subset X^{\vee}_{cusp}$.
\end{proposition}
\begin{proof}
Hessian matrix associated to $I^f$ is given by blocks $\left[B^{\lambda_i\lambda_j}_{ij}\right]$, $1\leq i,j\leq k$.
Let $J$ be the set with $I^f-J=\alpha$ and $J-I^f=\beta$. So intersection of $J$ and $I^f$ is of size $k-1$. With respect to $J$, $a^{\lambda_{\beta}}_{\beta}$ for all possible $\lambda_{\beta}$ is zero, since they are linear terms. Moreover $a^{\beta j^{\prime}}_{\alpha j}$ are zero in the Hessian matrix, since they are linear terms of $J$. For all possible $j$, $a^{\beta j^{\prime}}_{\alpha j}$ forms zero column and row in the Hessian matrix, so its determinant is zero. This implies $\nabla_{node}\left(J\right)\subset X^{\vee}_{cusp}$ if $\vert J\cap I^{f}\vert=k-1$.
\end{proof}

\subsection{Construction of $x(J,T)$}
For every $J$ and every $T\neq 0$ we define the point $x\left(J,T\right)$ as follows: 
Let $K$ be $k\times N$ matrix with the first $k\times k$ block is the identity matrix. Entries of the last $k\times k$ block are either $T, T^{-1}$ or zero described as:
\begin{itemize}
\item Put $T$ into the position $K_{rs}$, $r,s$ are row and column indices with $r\in I^f\cap J$, $s\in I^{\ell}-J$ is the replacement. 
\item For positions $I^{\ell}\cap J$, put $T^{-1}$. 
\item Fill the remaining positions with zeros. 
\end{itemize}

We give here some examples of points parametrized by $J,T$. 

\begin{example}
For $G\left(4,\Cc^8\right)$, given $J=\left\{1,6,7,8\right\}$.
\begin{align}
K=\begin{bmatrix}
1 &0 &0 &0 &T &0 &0 &0  \\
0 &1 &0 &0 &0 &T^{-1} &0 &0  \\
0 &0 &1 &0 &0 & 0& T^{-1}&0  \\
0 &0 &0 &1 &0 & 0& 0&T^{-1}  \\ 
\end{bmatrix} 
\end{align}

Replacement is between $1$ and $5$.
For $G\left(4,\Cc^{10}\right)$, given $J=\left(2,3,8,9\right)$
\begin{align}
K=\begin{bmatrix}
1 &0 &0 &0 &0 &0 &0 &T^{-1} &0 &0 \\
0 &1 &0 &0 &0 &0 &T &0 &0 &0 \\
0 &0 &1 &0 &0 &0 &0 &0 &0 &T \\
0 &0 &0 &1 &0 &0 &0 & 0&T^{-1} &0 \\ 
\end{bmatrix} 
\end{align}
Replacement is $1\mapsto 8$, $2\mapsto 7$, $3\mapsto 10$, $4\mapsto 9$.
For $G\left(5,\Cc^{14}\right)$, given $J=\left(2,4,12,13,15\right)$, last square block:
\begin{align}
K=\begin{bmatrix}
 0&T^{-1} &0 &0 &0  \\
 T&0 & 0& 0  & 0 \\
 0&0 & 0&  0& T^{-1}\\
 0&0 & 0 & T &0\\ 
 0& 0& T^{-1}&0 &0 \\ 
\end{bmatrix} 
\end{align}
\end{example}

\begin{lemma}\label{nodesubvarietynode}
If $\vert I^f\cap J\vert \leq k-3$ then 
\begin{align*}
\lim_{T\mapsto 0} \nabla\left(x\left(I^f\right)\right)\cap\nabla\left(x\left(J,T\right)\right)=\nabla\left(x\left(I^f\right)\right)\cap\nabla\left(x\left(J\right)\right),
\end{align*}
the limit taken in the Grassmannian of subspaces of codimension $2\left(k\left(N-k\right)+1\right)$.
\end{lemma}

\begin{proof}
We will exhibit a system of linear forms $\phi_{i,T}$, $\left(i=1,\ldots,2k\left(N-k\right)+2\right)$ on space $\cA$ defining the subspace $\nabla\left(x\left(I^f\right)\right)\cap\nabla\left(x\left(J,T\right)\right)$ and having the following property:\\
Each $\phi_{i,T}$ is a polynomial function of $T$ and the linear forms $\phi_{i,0}$ are linearly independent.\\
The above property implies existence of a limit operation, and limit is given by evaluation of forms at $T=0$.
The first $k\left(N-k\right)+1$ forms $\phi_{i,T}$ will be independent of $T$, these are the degree zero and degree one forms defining $\nabla\left(x\left(I^f\right)\right)$. \\
For the remaining forms we take forms associated to $x\left(J,T\right)$ and its partial derivatives. Explicitly:
\begin{align}
F\left(A,x\left(J,T\right)\right)=\left(\sum_{P\subset I^f} T^{\vert J\cap P\vert-\vert\overline{J}\cap P\vert}a_{r\left(P\right)}\right)
\end{align}
where $\overline{J}$ is the complement of $J$, $r\left(P\right)$ is replacement.
Then to make it polynomial we multiply it with suitable power of $T$, explicitly $T^{\vert I^f-J\vert}$.
\begin{example}
$G\left(4,\Cc^{10}\right)$, $J=\left(2,3,8,9\right)$, $I^f=\left(1,2,3,4\right)$, $I^{\ell}=\left(7,8,9,10\right)$, $\overline{J}=\left(1,4,7,10\right)$. Replacement is given by $1\mapsto 8$, $2\mapsto 7$, $3\mapsto 10$, $4\mapsto 9$. We give a few of them:
\begin{center}
 \begin{tabular}{| l | l | l  |}
    \hline
     $r\left(P\right)$  & $\overline{J}\cap P$ & $J\cap P$ \\ \hline
     $r\left(\left\{1\right\}\right)=8,2,3,4$ &1&0\\ \hline
     $r\left(\left\{2\right\}\right)=1,7,3,4$ &0&1\\ \hline
     $r\left(\left\{3\right\}\right)=1, 2, 10, 4$ &0&1\\ \hline
     $r\left(\left\{4\right\}\right)=1 ,2 ,3 ,9$ &1&0\\ \hline
     $r\left(\left\{1,2\right\}\right)=8, 7 ,3, 4$ &1&1\\ \hline
     $r\left(\left\{1,3\right\}\right)=8, 2, 10, 4$ &1&1\\ \hline
     $r\left(\left\{1 , 2,3\right\}\right)=8, 7, 10, 4$ &1&2\\ \hline    
    \end{tabular}
 \end{center}
\end{example}
The limit of this form is $a_{r\left(J\right)}$. 

All other forms 
\begin{align}
\frac{\partial F\left(A,x\left(J,T\right)\right)}{\partial x^j_i}
\end{align}

are treated in a similar way for possible indices $i,j$. They are all possible $1\leq j\leq k$ and $1\leq i\leq N$ excluding:
\begin{enumerate}[label=(\alph{*})]
\item $i=j\in I^f\cap J$, $i\in I^f$
\item $i\in T^l\cap J$, $j\in I^f-J$ and replacement of $j$ is $i$, $i\in I^{\ell}$.
\end{enumerate}

In the limit we obtain $a_{i_1,\ldots,i_k}$ for all $\left(i_1,\ldots,i_k\right)$ in the star of $J$. Those forms define $\nabla\left(x\left(J\right)\right)$, we are done.
\end{proof}

\begin{lemma}
If $\vert I^f\cap J\vert =k-2$ , $J-I^f=\left\{\alpha,\alpha'\right\}$, $I^f-J=\left\{t,t'\right\}$
\begin{align*}
\lim_{T\mapsto 0} \nabla\left(x\left(I^f\right)\right)\cap\nabla\left(x\left(J,T\right)\right)=\nabla\left(x\left(I^f\right)\right)\cap\nabla\left(x\left(J\right)\right),
\end{align*}
exists and is equal to the subspace in $\nabla\left(x\left(I^f\right)\right)\cap \nabla\left(x\left(J\right)\right)$ given by four additional equations:
\begin{align*}
\sum\limits_{j\in I^f\cap J} a^{r(j)\alpha}_{jt}=\sum\limits_{j\in I^f\cap J} a^{r(j)\alpha'}_{jt}=0 \\
\sum\limits_{j\in I^f\cap J} a^{r(j)\alpha}_{jt'}=\sum\limits_{j\in I^f\cap J} a^{r(j)\alpha'}_{jt'}=0 
\end{align*}
where $r(j)$ is the replacement of $j$.
\end{lemma}

\begin{proof}
We prove it in a similar way to the previous lemma but with a slight modification. When we normalize the form by multiplying a suitable power of $T$, we get
\begin{align*}
A\mapsto \frac{\partial}{\partial x^{t}_{\alpha}} F\left(A,x\left(J,T\right)\right)
\end{align*}
Let $\varphi^{'}_{T}$ be the form whose first two leading terms are given by:
\begin{align*}
a^{\alpha'}_{t'}+T\left(a+\sum\limits_{j\in I^f\cap J} a^{r(j)\alpha'}_{jt'}\right)
\end{align*}
Since the forms $a$ and $a^{\alpha'}_{t'}$ are degree zero and degree one terms with respect to $I^f$ which are already zero, we replace the form by
\begin{align*}
\varphi_{T}=T^{-1}\left(\varphi^{'}_{T}-Ta-a^{\alpha'}_{t'}\right)
\end{align*}
Now we have 
\begin{align}
\varphi_0=\sum\limits_{j\in I^f\cap J} a^{r(j)\alpha'}_{jt'}
\end{align}
We get the other equations by just differentiating the form at $x^{t'}_{\alpha}$, $x^{t'}_{\alpha'}$ and $x^{t}_{\alpha'}$ .\\

\begin{proposition}
For $G\left(k,\Cc^N\right)^{\vee}$, if $\vert J\cap I^f\vert\leq k-3$ then $\nabla_{node}\left(J\right)\subset\nabla_{node}\left(\emptyset\right)$.
\end{proposition}

\begin{theorem}
\begin{enumerate}[label=(\alph{*})]
\item Suppose that $\vert I^f\cap J\vert =k-1$. Then $\nabla_{node}(J)\subset X^{\vee}_{cusp}$
\item Suppose that $\vert I^f\cap J\vert \leq k-2$ for all $k,N$ except $(k,N)=(3,6),(3,7),(3,8),(3,9)$. Then $\nabla_{node}(J)\subset\nabla_{node}(\emptyset)$
\end{enumerate}
\end{theorem}

We already proved part a in \ref{nodesubvarietycusp}. Part b, if the cardinality intersection $I^f\cap J$ is smaller than $k-2$, it follows from lemma \ref{nodesubvarietynode}. In the equality case we have the following arguments:

To show that $\nabla_{node}(J)$ lies in the $\nabla_{node}(\emptyset)$, it suffices to show the existence of a dense Zariski open set $U\in \nabla\left(x\left(I^{\ell}\right)\right)\cap\nabla\left(x\left(J\right)\right) $ satisfying the conditions below:
\begin{itemize}
\item For every array $A\in U$, there exists a $g\in G$ such that 
\begin{align}
A\cdot g\in \nabla\left(x\left(I^{\ell}\right)\right)\cap\nabla\left(x\left(J\right)\right)
\end{align}
\item $A\cdot g$ satisfies the additional equations in the lemma.
\end{itemize}

We need to find $g\in G$ such that $g$ fixes every coordinate except the basis indexed by $I^{\ell}-J$. And it sends elements of $I^{\ell}-J$ to $\sum c^{j}_{i}e_{i}$. With this $A.g$ satisfying the above conditions, and the number of unknowns (i.e. $c^j_i$) is more than $4$, so there exist a solution with nonzero coefficients. Hence it is Zariski open and dense.
\end{proof}

\section{Decomposition of Singular Locus}\label{decompositionofsingularlocus}

In this section we prove that $X^{\vee}_{sing}=X^{\vee}_{node}\cup X^{\vee}_{cusp}$. To show this, we will modify the arguments of \cite{WZ} which follow the ideas of \cite{katz} and \cite{dimca}. 
Let $X^{\vee}_{sm}=X^{\vee}-X^{\vee}_{sing}$ be the set of smooth points of $X^{\vee}$ and $\cO$ be the complement of $X^{\vee}_{node}\cup X^{\vee}_{cusp}$ in $X^{\vee}$. The following result is due to N. Katz \cite{katz}:

\begin{proposition}[\cite{WZ}, Proposition 6.1] Suppose the projection $\pi_1:Z\mapsto X^{\vee}$ \ref{projection} is generically unramified. Then this projection is birational. Furthermore $\cO$ consists of smooth points of $X^{\vee}$ and is the biggest open set in $X^{\vee}$ for which the projection $\pi_1:\pi^{-1}_1(\cO)\mapsto \cO$ is an isomorphism.
\end{proposition}

$\pi_1:Z\mapsto X^{\vee}$ is generically ramified if and only if $X^{\vee}$ is hypersurface \cite{katz}. By this remark and proposition above, $\cO$ is a subvariety of $X^{\vee}_{sm}$, which is equivalent to:
\begin{align}\label{firstinclusion}
X^{\vee}_{sing}\subseteq X^{\vee}_{node}\cup X^{\vee}_{cusp}
\end{align}

Now we need to show the reverse inclusion. We will deduce it from the following two propositions:

\begin{proposition}[\cite{WZ} Lemma 6.2] The variety of arrays $A$ having infinitely many critical points in $X^{\vee}$ is of codimension $\geq 2$ in $X^{\vee}$.\label{WZlemma} 
\end{proposition}This lemma is valid for any smooth projective variety such that its projectively dual is hypersurface.
The next result is due to A. Dimca \cite{dimca}:
\begin{proposition}[\cite{WZ}, Proposition 6.3.]\label{dimcalemma}
Suppose an array $A\in X^{\vee}$ has finitely many critical points in $X$. Let $H_A$ be the hyperplane in $\Pp(\bigwedge^k\Cc^N)$ defined by $A$. Then the multiplicity of $X^{\vee}$ at $A$ is equal to
\begin{align}
mult_A(X^{\vee})=\sum\limits_{x\in X(A)} \mu(X\cap H_A,x)
\end{align}
where the sum is over all critical points of $A$ in $X$ and $\mu(X\cap H_A,x)$ is the Milnor number of $X\cap H_A$ at $x$. In particular $\mu(X\cap H_A,x)=1$ if and only if the Hessian of $A$ at $x$ is nondegenerate.
\end{proposition}
We see that for all $\bigwedge^k\Cc^N$ except $\bigwedge^3\Cc^6$, $\bigwedge^3\Cc^7$,$\bigwedge^3\Cc^8$, all of the irreducible components of $X^{\vee}_{node}$ and $X^{\vee}_{cusp}$ are of codimension one in $X^{\vee}$. By proposition \ref{WZlemma}, the generic point $A$ of each of them has finitely many critical points in $X$. By \ref{dimcalemma} the multiplicity of $X^{\vee}$ at $A$ is $\geq 2$ which implies $A$ is a singular point of $X^{\vee}$. Therefore $X^{\vee}_{node}\cup X^{\vee}_{cusp}\subseteq X^{\vee}_{sing}$. We combine it with \ref{firstinclusion} to get:
\begin{align}
X^{\vee}_{sing}=X^{\vee}_{node}\cup X^{\vee}_{cusp}
\end{align}
For the remaining cases, since $X^{\vee}_{node}\subset X^{\vee}_{cusp}$, we get $X^{\vee}_{sing}=X^{\vee}_{cusp}$. Also one can check this by just looking at their orbits \cite{KW1},\cite{KW2}, \cite{hol}.

\begin{theorem}\label{cuspnodedifferent}
The cusp and the generic node components are different for the dual Grassmannians except $k=3,N=6,7,8$. Therefore we obtain nontrivial decomposition:
$X^{\vee}_{sing}=X^{\vee}_{node}\cup X^{\vee}_{cusp}$.
\end{theorem}
\begin{proof}
To prove $\nabla_{node}(\emptyset)\neq X^{\vee}_{cusp}$ in $\cA$, we modify arguments used in \cite{WZ}. We will construct conormal bundles of cusp and node component, and show that their generic fibers are different. We can identify the space of arrays $\cA$ with $\bigwedge^k(\Cc^{\ast})^N$. Let $e_1,e_2,\ldots,e_N$ be a basis of $(\Cc^{\ast})^N$. Now we can make a pairing between $da_{i_1\ldots i_k}$ and $e_{i_1}\wedge\cdots\wedge e_{i_k}$. We deduce the following:
\begin{itemize}
\item For a generic $A\in\nabla(x^0)\cap\nabla(x')$, the conormal space $T_{A,X^{\vee}_{node}}\cA$ is the two dimensional subspace in $\cA$ spanned by $e_1\wedge\cdots\wedge e_k$ and $e_{N-k+1}\wedge\cdots\wedge e_{N}$. By the group action, the generic fiber of the bundle $T_{X^{\vee}_{node}}\cA$ is spanned by two vectors $u_1\wedge\cdots\wedge u_k$ and $v_1\wedge\cdots\wedge v_k$ where all $u_i,v_j$ are linearly independent.
\item By similar arguments the generic fiber of the conormal bundle $T_{X^{\vee}_{cusp}}\cA$ is a two dimensional vector space of the form
\begin{align}
W=\sum\limits^{k}_{j=1} w_1\wedge\cdots\wedge w_{j-1}\wedge(\Cc^*)^N\wedge w_{j+1}\wedge\cdots\wedge w_{k}
\end{align} 
where $w_j\in(\Cc^*)^N$ and nonzero for all $j=1,\ldots,k$.
\end{itemize}
For any two alternating decomposable tensors  $u_1\wedge\cdots\wedge u_k$ and $v_1\wedge\cdots\wedge v_k$ contained in $W$, there is an index $j$ such that $u_j$ is proportional to $v_j$. Hence $W$ cannot contain a generic fiber of $T_{X^{\vee}_{node}}\cA$. This implies that $X^{\vee}_{node}\neq X^{\vee}_{cusp}$.
\end{proof}

\begin{remark} In \cite{landsberg}, the relationship between secant varieties and the node type subvarieties is given by:
\begin{align}
X^{\vee}_{node,j}=\sigma_j(X)^{\vee}
\end{align}
where $\sigma_j$ is the $j$th secant variety, and $X^{\vee}_{node,j}$ is the Zariski closure of the points of $X^{\vee}$ tangent to $X$ at least $j$ points. Moreover,  the decomposition  $X^{\vee}=X^{\vee}_{cusp}\cup X^{\vee}_{node,2}$ is by definition \cite{landsberg}, p220. In \cite{hol}, dimension of the secant variety $\sigma_2(X)$ was computed to derive the result \ref{resultnode}.
\end{remark}

\section{Symmetric Matrices with Diagonal Lacunae and Skew Symmetric Blocks}\label{hessian}

We consider the following type of matrices: 
 for given $k\geq 2$, and $N$, we construct a $k\left(N-k\right)\times k\left(N-k\right)$ matrix $H_{k,N}$ which has the block decomposition with
$k\times k$ skew symmetric matrices of size $\left(N-k\right)\times\left( N-k\right)$.

In other words, they have the following shape:
\begin{align}
H_{k,N}=\begin{bmatrix}
0 & A_{12} & \cdots &\cdots & A_{1k} \\
-A_{12} & 0 & \cdots &\cdots & A_{2k} \\ 
\vdots & \vdots & \ddots & &\vdots \\
\vdots&\vdots& &0 & A_{k-1\,k} \\
-A_{1k} & -A_{2k} & \cdots & -A_{k-1\,k} &0 
\end{bmatrix} 
\end{align}

where each $A_{ij}$ and $0$ are skew symmetric blocks of size $\left(N-k\right)\times\left(N-k\right)$. Notice that the Hessian matrices of dual Grassmannian is of this shape by simple computation of \ref{h1}.

\begin{theorem}\label{determinant} Let $H_{k,N}$ be the Hessian matrix of the dual Grassmanian $G\left(k,\Cc^N\right)^{\vee}$ evaluated at $x^0$ i.e. $H_{k,N}=H(F(A,x))_{x^0}$. Then we have the following:
the determinant of $H_{k,N}$ is an irreducible homogeneous polynomial for all $N\geq 2k$, $k\geq 3$ except $k=3, N=6,7$. In these cases we have:
\begin{enumerate}[label=(\arabic{*})]
\item $\det H_{3,6}$ is the cube of degree $3$ irreducible polynomial.
\item  $\det H_{3,7}$ is the square of degree $6$ irreducible polynomial.
\end{enumerate}
\end{theorem}

\begin{remark} The dual variety is hypersurface if and only if there exists an invertible Hessian matrix at a smooth point \cite{GKZ}, \cite{WZ2}. So, by the theorem \ref{determinant}, for all $N\geq 2k$, $k\geq 3$, we conclude that the dual Grassmannian is  hypersurface in $\Pp\left(\bigwedge^k\Cc^N\right)$.
\end{remark}

\begin{remark} The determinant of the Hessian matrix $H_{3,6}$ \ref{example}  has a nice description: 
\begin{align}
\det H_{3,6}=2\det^{\quad\, 3}\begin{vmatrix}
a^{45}_{12} & a^{46}_{12} &a^{56}_{12}  \\
a^{45}_{13} & a^{46}_{13} & a^{56}_{13} \\ 
a^{45}_{23} &a^{46}_{23} & a^{56}_{23}
\end{vmatrix}
\end{align}
\end{remark}

We prove theorem \ref{determinant} in steps, first we study the case $k=3$.
\subsection{$\det H_{3,N}$}
Since $k=3$, we have the following $3\times 3$ block matrix
\begin{center}
$H_{3,N}=\left[\begin{matrix}
0&A_{12}&A_{13}\\
-A_{12}&0&A_{23}\\
-A_{13}&-A_{23}&0\end {matrix}\right] $ .
\end{center}

To analyze its determinant we use Young diagrams. Consider the direct sum decomposition of $\Cc^N$ to $\Cc^3\oplus\Cc^{N-3}$. Entries of the Hessian matrix can be identified by the quadratic part of 
 $\bigwedge^3 \Cc^N$ which is $\Cc^3\otimes\bigwedge^2 \Cc^{N-3}$. 
 
The existence of an invariant of degree $d$ is equivalent to the existence of rectangular diagrams of size $d$ in the decomposition of $Sym_d\left(\Cc^3\otimes\bigwedge^2\Cc^{N-3}\right)$. We apply Cauchy's formula to get:
\begin{align}
Sym_d\left(\Cc^{3}\otimes\bigwedge^{2}
     \Cc^{N-3}\right)=\bigoplus_{\mid\lambda\mid\vdash d} S_{\lambda}\Cc^{3}\otimes S_{\lambda}\left(\bigwedge^{2}\Cc^{N-3}\right) 
\end{align}
where $\lambda=(\lambda_1\geq \lambda_2\geq \lambda_3)$ is Young diagram of height  $ht \left(\lambda\right)\leq 3$. Otherwise, the representation is $S_{\lambda}\Cc^3\cong 0$.

Existence of rectangular partitions impose the arithmetic conditions:

\begin{align*}
 3\mid d, &\hspace{10pt} (N-3) \mid 2d, \hspace{10pt} d\leq 3\left(N-3\right)
\end{align*}
The last condition arose since we are looking for possible polynomial factors of the determinant and the degree of it is $3\left(N-3\right)$. If $N-3$ is a prime number greater than $3$, $\det H_{3,N}$ is irreducible, since $d$ becomes the product of two primes. This implies $\det H_{3,8}$ and $\det H_{3,10}$ are irreducible polynomials of degrees $15$ and $21$ respectively. We obtain the factors of exceptional cases and irreducibility of $\det H_{3,9}$ by using MAPLE. \\
Before the analysis of the other cases, we point out that by the canonical isomorphism $G\left(k,\Cc^N\right)\cong G\left(N-k,(\Cc^N)^{\ast}\right)$, we get $H_{k,N}=H_{N-k,N}$. We give a direct computational proof:

\begin{proposition}
The Hessian matrix of type $H_{k,N}$ can be obtained by row and column changes on $H_{N-k,k}$. 
\end{proposition}

\begin{proof}
For the matrix $H_{3,N}$, arrange rows and columns in this order:\\ 
$\left[1, N-2, 2N-5\vert 2,N-1,2N-4\vert\ldots\vert N-3,2N-6,3N-9\right]$. It gives the block structure of $H_{N-k,N}$.
For general $k$ and $N$, the row and the column changes below gives the block structure of $H_{N-k,N}$:
\begin{gather*}
\left[ 1, N-k+1, 2\left(N-k\right)+1\ldots, \left(k-1\right)\left(N-k\right)+1\right] \\
\left[ 2,N-k+2,2\left(N-k\right)+2\ldots, \left(k-1\right)\left(N-k\right)+2\right] \\
\vdots \\
\left[ N-k,2\left(N-k\right),\ldots,k\left(N-k\right)\right]
\end{gather*}
\end{proof}

Let $S$ be a skew symmetric matrix of size $k\times k$. We can express it in terms of smaller blocks i.e. 
\begin{align}\label{skewsym}
 S = 
 \begin{bmatrix}
S_1& U \\ 
-U^t& S_2 \\ 
 \end{bmatrix}
 \end{align}
where the blocks on diagonal $S_i$ are skew symmetric of size $k_1,k_2$ with $2\leq k_1,k_2$ and $k_1+k_2=k$ and $U$ is the $k_1\times k_2$ rectangular block with arbitrary entries, $-U^t$ is the $k_2\times k_1$ block which is the transpose of $U$ multiplied by $-1$. \\
We want to make a similar block decomposition to $H_{3,N}$. If the shape of $H_{3,N}$ is 
\begin{center}
$H_{3,N}=\left[\begin{matrix}
0^{\prime}&A^{\prime}&B^{\prime}\\
-A^{\prime}&0^{\prime}&C^{\prime}\\
-B^{\prime}&-C^{\prime}&0^{\prime}\end {matrix}\right] $ .
\end{center}
where $A^{\prime},B^{\prime},C^{\prime}$ are skew symmetric matrices of size $\left(N-3\right)\times\left(N-3\right)$, $0^{\prime}$ are the zero matrices matrix of size $\left(N-3\right)\times\left(N-3\right)$ then we can divide $A^{\prime},B^{\prime},C^{\prime}$  into blocks described in  (\ref{skewsym}) above. Now $H_{3,N}$ becomes :

\begin{center}
\[H_{3,N}=
  \begin{bmatrix}
    \begin{array}{cc|cc|cc}
0&0&A_1&A&B_1&B\\
0&0&-A^t&A_2&-B^t&B_2\\\hline
-A_1&-A&0&0&C_1&C\\
A^t&-A_2&0&0&-C^t&C_2\\\hline
-B_1&-B&-C_1&-C&0&0\\
B^t&-B_2&C^t&-C_2&0&0\\
 \end{array}
  \end{bmatrix}
\].
\end{center}

with sizes $A_1,B_1,C_1=k_1\times k_1$, $A_2,B_2,C_2=k_2\times k_2$, $k_1+k_2=N-3$, $k_1\geq k_2$. After row and column operations $\left(2,4,5,3\right)$ applied to blocks we get:

\begin{align}\label{m3}
H_{3,N}=
  \begin{bmatrix}
    \begin{array}{ccc|ccc}
  0&A_1&B_1&0&A&B\\
-A_1&0&C_1&-A&0&C\\
-B_1&-C_1&0&-B&-C&0\\\hline
0&-A^t&-B^t&0&A_2&B_2\\
A^t&0&-C^t&-A_2&0&C_2\\
B^t&C^t&0&-B_2&-C_2&0\\
    \end{array}
  \end{bmatrix}
\end{align}

The first and the last $3\times 3$ blocks on the diagonal are in the form of $H_{3,k_1+3}$ and $H_{3,k_2+3}$. If we call the upper right $3\times 3$ blocks  $U$, the complementary block is the transpose of $U$. \\

Before the analysis of determinants, we state a basic tool which we do not know any reference for it:
\begin{lemma}\label{keylemma}
Let $P$ be a homogeneous multivariate polynomial. Suppose that by setting some variables to zero we get $P'=P_1\times P_2$, $\deg P_1\geq \deg P_2$ and $P_1,P_2$ are irreducible homogeneous polynomials. Similarly, suppose that by setting some variables to zero  $P''=Q_1\times Q_2$, $\deg Q_1\geq \deg Q_2$ and $Q_1,Q_2$ are irreducible homogeneous polynomials. If $\deg P_1>\deg Q_1 $ and $\deg P_1\neq \deg Q_2$, then $P$ has to be an irreducible polynomial.
\end{lemma}
\begin{proof}
Assume to the contrary that $P$ has nonconstant polynomial factors. Specialization with zero cannot increase the degrees of those irreducible factors. Thus, the polynomial factors assumed in the lemma can only appear if $P$ is irreducible.
\end{proof}

After setting $A,B,C$ to zero in \ref{m3} we have $H_{3,k_1+3}\oplus H_{3,k_2+3}\mapsto H_{3,N}$, where the matrix $H_{3,N}$ has the zero blocks. We call this method \emph{specialization} and it allows us to embed smaller matrices into larger ones.

\begin{lemma}
If $N\geq 11$, $\det H_{3,N}$ is an irreducible polynomial.
\end{lemma}
\begin{proof}
Consider the specializations:
\begin{align}\label{g3embed}
H_{3,6}\oplus H_{3,N-3}\mapsto H_{3,N}\\
H_{3,7}\oplus H_{3,N-4}\mapsto H_{3,N}
\end{align}
For small $N$'s we give the possible polynomial factors of these embeddings, then we derive a contradiction. By the lemma \ref{keylemma}, it is enough to analyze the degrees of possible polynomial factors of the determinant. If they are not compatible with each other we conclude that it has to be an irreducible polynomial.
First we analyze $H_{3,11}$. We can embed $H_{3,6}$ and $H_{3,8}$ into $H_{3,11}$ and $H_{3,7}\oplus H_{3,7}$ as in \ref{g3embed}. Degrees of irreducible factors are $3+3+3+15$ and $6+6+6+6$ which are not compatible by lemma \ref{keylemma}, hence $\det H_{3,11}$ is irreducible polynomial of degree $24$.\\
Now we can use proof by induction. We have showed that $\det H_{3,8}, \det H_{3,10}$,$\det H_{3,11}$ are irreducible. Assume that for all $11\leq n<N$ $\det H_{3,n}$ is irreducible, and they have degree $3\left(n-3\right)$. If we apply embedding \ref{g3embed}, we get possible factors of $\det H_{3,N}$ as $3+3+3+3\left(N-6\right)$ and $6+6+3\left(N-7\right)$ and they are not compatible by lemma \ref{keylemma}. This forces that $\det H_{3,N}$ is irreducible.

\end{proof}

\subsection{$\det H_{4,N}$}
Now we focus on $H_{4,N}$, $N\geq 8$. The matrix we are dealing with is:
\begin{align}\label{Hess4N}
\left[\begin{matrix}
0&A_{12}&A_{13}&A_{14}\\
-A_{12}&0&A_{23}&A_{24}\\
-A_{13}&-A_{23}&0&A_{34}\\
-A_{14}&-A_{24}&-A_{34}&0\end {matrix}\right] 
\end{align}

Observe that if we set $A_{34}$ to be the zero matrix in (\ref{Hess4N}), $\det H_{4,N}$ is the square of the irreducible determinant. Together with lemma \ref{keylemma} this implies that $\det H_{4,N}$ has at most two factors of the same degree.

\begin{lemma}
If $N\geq 8$, $\det H_{4,N}$ is irreducible
\end{lemma}

\begin{proof}
By computer computation we verify that, $\det H_{4,8}$ and $\det H_{4,9}$ are irreducible.\\
Now we modify the method we used for $H_{3,N}$. For each $A_{ij}$ of matrix \ref{Hess4N}, we apply \ref{skewsym}. This allows us to embed smaller Hessian matrices into larger ones. It is enough to use the following specializations:

\begin{align}\label{g4embed1}
H_{4,6}\oplus H_{4,N-2}\mapsto H_{4,N}\\\label{g4embed2}
H_{4,7}\oplus H_{4,N-3}\mapsto H_{4,N}
\end{align}
For $N=10$, degrees of possible factors are $2+2+2+2+16$ and $6+6+6+6$. By \ref{keylemma}, they are compatible only $\det H_{4,10}$ is irreducible. \\
Now we can use induction. Suppose that for all $10\leq n<N$, $\det H_{4,n}$ is irreducible. By \ref{g4embed1} and \ref{g4embed2}, degrees of possible factors are $2+2+2+2+4\left(N-6\right)$ and $6+6+4\left(N-7\right)$. They are compatible if $\det H_{4,N}$ is irreducible. 
\end{proof}

\subsection{$\det H_{k,N}$, $k\geq 5$}
We start with the case $k=5$. We obtain irreducibility of $\det H_{5,10}$ by computer computation. For $H_{5,11}$, we can embed $H_{5,8}\oplus H_{5,8}$ and $H_{3,9}\oplus H_{2,8}$ by using duality and setting $A_{i,5}$, $i\leq 4$ to zero
respectively. Degrees of possible factors are $15+15$ and $18+3+3+3+3$, by lemma \ref{keylemma} we obtain irreducibility of $\det H_{5,11}$.\\
In the case of $H_{5,12}$, we can use Young diagrams. A degree $d$ invariant exists if and only if there exists rectangular diagrams satisfying $5\mid 3d$ and $7\mid 2d$. This implies $d=35$ which is the degree of determinant. 
For the remaining cases we can use induction together with the specializations:
\begin{align}
H_{5,8}\oplus H_{5,N-3}\mapsto H_{5,N}\\
H_{5,9}\oplus H_{5,N-4}\mapsto H_{5,N}
\end{align}

In general we can use the following specializations:
\begin{align}
H_{k,k+2}\oplus H_{k,N-2}\mapsto H_{k,N}\\
H_{k,k+3}\oplus H_{k,N-3}\mapsto H_{k,N}\\
H_{k,k+4}\oplus H_{k,N-4}\mapsto H_{k,N}
\end{align}
and induction to show that $\det H_{k,N}$ is irreducible depending on whether $k$ is odd or even. If $k$ is odd, $\det H_{k,k+2}$ is identically zero. The only uncovered case is $H_{7,14}$, we obtain its irreducibility by computer computation.

\subsection{Rank Lemmas} \label{ranklemmas}
In the proof of theorem \ref{thmcusp}, we used existence of Hessian matrices $H_{k,N}$ satisfying:
\begin{enumerate}[label=(\roman{*})]
\item It has corank $1$. 
\item Each row block of size $k\times k\left(N-k\right)$ is of full rank.
\end{enumerate}

Observe that, if there exist a matrix $A$ satisfying the above, then $A\oplus A^{\prime}$ satisfy it when $A'$ is of full rank. It is enough to produce those matrices for $(k,N)\in\left\{(3,8),(3,9),(3,10),(3,11), (4,8), (4,9), (5,10)\right\}$ because we can use the following specializations inductively:
\begin{enumerate}[label=(\arabic{*})]
\item $H_{3,6}\oplus H_{3,a}\hookrightarrow H_{3,a+3}$, the first matrix is of full rank.
\item $H_{4,6}\oplus H_{4,a}\hookrightarrow H_{4,a+2}$, the first matrix is of full rank.
\item $H_{5,8}\oplus H_{5,a}\hookrightarrow H_{5,a+3}$, the first matrix is of full rank.
\item $H_{3,N-k+3}\oplus H_{k-3,N-3}\hookrightarrow H_{k,N}$, the second matrix is of full rank. 
\end{enumerate}

\subsection{List of Matrices 1}
Here we give the list of matrices which satisfies the above conditions  \ref{ranklemmas}. We used them to prove the theorem \ref{thmcusp}.
\newpage
\begin{changemargin}{-3cm}{-3cm}
\begin{center}
$\bullet\bigwedge^3\Cc^9$\\\vspace{10pt}
$A_{12}\!=\!\left[ \begin {array}{cccccc} 0&1&0&0&1&0\\\noalign{\medskip}-1&0&1&0&1&1\\\noalign{\medskip}0&-1&0&0&1&0\\\noalign{\medskip}0&0&0&0&0&0
\\\noalign{\medskip}-1&-1&-1&0&0&0\\\noalign{\medskip}0&-1&0&0&0&0
\end {array} \right] 
$,
$A_{13}\!=\!\left[ \begin {array}{cccccc} 0&0&1&1&0&1\\\noalign{\medskip}0&0&0&1&0&1\\\noalign{\medskip}-1&0&0&0&1&0\\\noalign{\medskip}-1&-1&0&0&1&0
\\\noalign{\medskip}0&0&-1&-1&0&1\\\noalign{\medskip}-1&-1&0&0&-1&0
\end {array} \right] 
$,
$A_{23}\!=\!\left[ \begin {array}{cccccc} 0&1&1&0&1&1\\\noalign{\medskip}-1&0&0&1&1&1\\\noalign{\medskip}-1&0&0&0&0&0\\\noalign{\medskip}0&-1&0&0&0&0
\\\noalign{\medskip}-1&-1&0&0&0&1\\\noalign{\medskip}-1&-1&0&0&-1&0
\end {array} \right] 
$
\end{center}
\end{changemargin}
\hrulefill
\begin{changemargin}{-3.75cm}{-3.75cm}

\begin{center}
$\bullet\bigwedge^3\Cc^{10}$\\\vspace{10pt}
$A_{12}\!=\!\left[ \begin {array}{ccccccc} 0&0&0&1&1&0&0\\\noalign{\medskip}0&0&0&1&0&0&0\\\noalign{\medskip}0&0&0&1&0&1&1\\\noalign{\medskip}-1&-1&-1&0
&1&0&1\\\noalign{\medskip}-1&0&0&-1&0&1&0\\\noalign{\medskip}0&0&-1&0&
-1&0&1\\\noalign{\medskip}0&0&-1&-1&0&-1&0\end {array} \right]$\!\! $A_{13}\!=\!\left[ \begin {array}{ccccccc} 0&1&0&0&0&1&1\\\noalign{\medskip}-1&0&1&0&0&1&1\\\noalign{\medskip}0&-1&0&0&0&0&0\\\noalign{\medskip}0&0&0&0
&1&1&0\\\noalign{\medskip}0&0&0&-1&0&0&0\\\noalign{\medskip}-1&-1&0&-1
&0&0&1\\\noalign{\medskip}-1&-1&0&0&0&-1&0\end {array} \right]$\!\! $A_{23}\!=\!\!\left[ \begin {array}{ccccccc} 0&0&0&0&1&1&1\\\noalign{\medskip}0&0&0&1&0&0&0\\\noalign{\medskip}0&0&0&1&0&0&0\\\noalign{\medskip}0&-1&-1&0
&0&0&1\\\noalign{\medskip}-1&0&0&0&0&1&1\\\noalign{\medskip}-1&0&0&0&-
1&0&0\\\noalign{\medskip}-1&0&0&-1&-1&0&0\end {array} \right] $
\end{center}
\end{changemargin}
\hrulefill
\begin{changemargin}{-2.5cm}{-2.5cm}

\begin{center}
$\bullet\bigwedge^3\Cc^{11}$\\\vspace{20pt}
$A_{12}=\left[ \begin {array}{cccccccc} 0&1&0&0&1&1&0&0\\\noalign{\medskip}-1&0&0&1&1&1&1&1\\\noalign{\medskip}0&0&0&0&0&1&1&0\\\noalign{\medskip}0
&-1&0&0&1&1&0&1\\\noalign{\medskip}-1&-1&0&-1&0&0&1&1
\\\noalign{\medskip}-1&-1&-1&-1&0&0&0&1\\\noalign{\medskip}0&-1&-1&0&-
1&0&0&0\\\noalign{\medskip}0&-1&0&-1&-1&-1&0&0\end {array} \right] $,\hspace{20pt}
$A_{13}= \left[ \begin {array}{cccccccc} 0&1&1&0&0&1&1&0\\\noalign{\medskip}-1
&0&0&1&1&0&0&1\\\noalign{\medskip}-1&0&0&0&1&0&1&1\\\noalign{\medskip}0
&-1&0&0&0&0&0&1\\\noalign{\medskip}0&-1&-1&0&0&1&0&0
\\\noalign{\medskip}-1&0&0&0&-1&0&1&1\\\noalign{\medskip}-1&0&-1&0&0&-
1&0&0\\\noalign{\medskip}0&-1&-1&-1&0&-1&0&0\end {array} \right] $\\\vspace{10pt}
$A_{23}=\left[ \begin {array}{cccccccc} 0&1&0&0&1&1&0&0\\\noalign{\medskip}-1&0&0&0&0&1&1&1\\\noalign{\medskip}0&0&0&0&0&1&0&1\\\noalign{\medskip}0
&0&0&0&1&1&1&1\\\noalign{\medskip}-1&0&0&-1&0&0&1&0
\\\noalign{\medskip}-1&-1&-1&-1&0&0&0&1\\\noalign{\medskip}0&-1&0&-1&-
1&0&0&0\\\noalign{\medskip}0&-1&-1&-1&0&-1&0&0\end {array} \right]$
\end{center}
\end{changemargin}

\begin{center}
$\bullet\bigwedge^4\Cc^8$\\\vspace{20pt}
$\left[ \begin {array}{cccc} 0&A_{{12}}&A_{{13}}&A_{{14}}\\\noalign{\medskip}-A_{{12}}&0&A_{{23}}&A_{{24}}
\\\noalign{\medskip}-A_{{13}}&-A_{{23}}&0&A_{{34}}
\\\noalign{\medskip}-A_{{14}}&-A_{{24}}&-A_{{34}}&0\end {array}
 \right] 
$\\\vspace{20pt}
$A_{12}=\left[ \begin {array}{cccc} 0&0&1&1\\\noalign{\medskip}0&0&0&1\\\noalign{\medskip}-1&0&0&0\\\noalign{\medskip}-1&-1&0&0\end {array}
 \right]$,\hspace{10pt}
$A_{13}=\left[ \begin {array}{cccc} 0&0&0&0\\\noalign{\medskip}0&0&0&1\\\noalign{\medskip}0&0&0&1\\\noalign{\medskip}0&-1&-1&0\end {array}
 \right] $,\hspace{10pt}
$A_{14}=\left[ \begin {array}{cccc} 0&1&0&1\\\noalign{\medskip}-1&0&0&1\\\noalign{\medskip}0&0&0&0\\\noalign{\medskip}-1&-1&0&0\end {array}
 \right] $\\\vspace{20pt}
$A_{23}=\left[ \begin {array}{cccc} 0&0&1&0\\\noalign{\medskip}0&0&1&1\\\noalign{\medskip}-1&-1&0&0\\\noalign{\medskip}0&-1&0&0\end {array}
 \right] $,\hspace{10pt}
 $A_{24}=\left[ \begin {array}{cccc} 0&0&0&1\\\noalign{\medskip}0&0&0&1\\\noalign{\medskip}0&0&0&1\\\noalign{\medskip}-1&-1&-1&0\end {array}
 \right] $,\hspace{10pt}
 $A_{34}=\left[ \begin {array}{cccc} 0&0&0&0\\\noalign{\medskip}0&0&0&0\\\noalign{\medskip}0&0&0&0\\\noalign{\medskip}0&0&0&0\end {array}
 \right] $
\end{center}

\hrulefill

\begin{changemargin}{-2.5cm}{-2.5cm}
\begin{center}
$\bullet\bigwedge^4\Cc^9$\\\vspace{20pt}
$A_{12}= \left[ \begin {array}{ccccc} 0&0&1&1&0\\\noalign{\medskip}0&0&1&0&1
\\\noalign{\medskip}-1&-1&0&0&1\\\noalign{\medskip}-1&0&0&0&0
\\\noalign{\medskip}0&-1&-1&0&0\end {array} \right] 
$,\hspace{10pt}$A_{13}=\left[ \begin {array}{ccccc} 0&1&1&1&1\\\noalign{\medskip}-1&0&0&1&0\\\noalign{\medskip}-1&0&0&1&0\\\noalign{\medskip}-1&-1&-1&0&0
\\\noalign{\medskip}-1&0&0&0&0\end {array} \right] 
$,\hspace{10pt}$A_{14}=\left[ \begin {array}{ccccc} 0&0&0&0&0\\\noalign{\medskip}0&0&1&0&1\\\noalign{\medskip}0&-1&0&0&1\\\noalign{\medskip}0&0&0&0&0
\\\noalign{\medskip}0&-1&-1&0&0\end {array} \right] 
$\\\vspace{20pt}
$A_{23}=\left[ \begin {array}{ccccc} 0&0&0&1&1\\\noalign{\medskip}0&0&0&0&0\\\noalign{\medskip}0&0&0&1&1\\\noalign{\medskip}-1&0&-1&0&1
\\\noalign{\medskip}-1&0&-1&-1&0\end {array} \right] 
$,\hspace{10pt}
$A_{24}=\left[ \begin {array}{ccccc} 0&1&1&0&0\\\noalign{\medskip}-1&0&0&0&1\\\noalign{\medskip}-1&0&0&0&1\\\noalign{\medskip}0&0&0&0&0
\\\noalign{\medskip}0&-1&-1&0&0\end {array} \right] 
$,\hspace{10pt}
$A_{34}=\left[ \begin {array}{ccccc} 0&1&1&1&0\\\noalign{\medskip}-1&0&1&0&1\\\noalign{\medskip}-1&-1&0&0&1\\\noalign{\medskip}-1&0&0&0&0
\\\noalign{\medskip}0&-1&-1&0&0\end {array} \right]$
\end{center}
\end{changemargin}

\hrulefill
\begin{changemargin}{-4cm}{-4cm}
\begin{center}
$\bullet\bigwedge^5\Cc^{10}$,
$\left[ \begin {array}{ccccc} 0&A_{{12}}&A_{{13}}&A_{{14}}&A_{{15}}\\\noalign{\medskip}-A_{{12}}&0&A_{{23}}&A_{{24}}&A_{{25}}
\\\noalign{\medskip}-A_{{13}}&-A_{{23}}&0&A_{{34}}&A_{{35}}
\\\noalign{\medskip}-A_{{14}}&-A_{{24}}&-A_{{34}}&0&A_{{45}}
\\\noalign{\medskip}-A_{{15}}&-A_{{25}}&-A_{{35}}&-A_{{45}}&0
\end {array} \right] 
$\\\vspace{20pt}
$A_{12}=\left[ \begin {array}{ccccc} 0&-1&-1&0&-1\\\noalign{\medskip}1&0&0&-1&-1\\\noalign{\medskip}1&0&0&-1&-1\\\noalign{\medskip}0&1&1&0&-1
\\\noalign{\medskip}1&1&1&1&0\end {array} \right] 
$,\hspace{10pt}
$A_{13}=\left[ \begin {array}{ccccc} 0&0&-1&-1&-1\\\noalign{\medskip}0&0&-1&-1&-1\\\noalign{\medskip}1&1&0&0&-1\\\noalign{\medskip}1&1&0&0&-1
\\\noalign{\medskip}1&1&1&1&0\end {array} \right] 
$,\hspace{10pt}
$A_{14}=\left[ \begin {array}{ccccc} 0&0&0&-1&0\\\noalign{\medskip}0&0&-1&-1&-1\\\noalign{\medskip}0&1&0&0&-1\\\noalign{\medskip}1&1&0&0&0
\\\noalign{\medskip}0&1&1&0&0\end {array} \right] 
$ \\\vspace{10pt}
$A_{15}=\left[ \begin {array}{ccccc} 0&0&-1&0&0\\\noalign{\medskip}0&0&0&0&0\\\noalign{\medskip}1&0&0&-1&0\\\noalign{\medskip}0&0&1&0&-1
\\\noalign{\medskip}0&0&0&1&0\end {array} \right]$,\hspace{10pt}
 $A_{23}=\left[ \begin {array}{ccccc} 0&0&0&-1&0\\\noalign{\medskip}0&0&-1&-1&-1\\\noalign{\medskip}0&1&0&-1&0\\\noalign{\medskip}1&1&1&0&-1
\\\noalign{\medskip}0&1&0&1&0\end {array} \right] 
$,\hspace{10pt}
 $A_{24}=\left[ \begin {array}{ccccc} 0&0&-1&0&-1\\\noalign{\medskip}0&0&-1&0&-1\\\noalign{\medskip}1&1&0&0&0\\\noalign{\medskip}0&0&0&0&-1
\\\noalign{\medskip}1&1&0&1&0\end {array} \right] 
$\\\vspace{10pt}
 $A_{25}\!=\!\! \left[ \begin {array}{ccccc} 0&0&-1&0&0\\\noalign{\medskip}0&0&0&-1&0
\\\noalign{\medskip}1&0&0&-1&0\\\noalign{\medskip}0&1&1&0&-1
\\\noalign{\medskip}0&0&0&1&0\end {array} \right] 
$,
 $A_{34}\!=\!\!  \left[ \begin {array}{ccccc} 0&0&-1&0&-1\\\noalign{\medskip}0&0&0&0&0
\\\noalign{\medskip}1&0&0&-1&0\\\noalign{\medskip}0&0&1&0&-1
\\\noalign{\medskip}1&0&0&1&0\end {array} \right] 
$,
 $A_{35}\!=\!\! \left[ \begin {array}{ccccc} 0&0&0&0&-1\\\noalign{\medskip}0&0&0&0&0
\\\noalign{\medskip}0&0&0&0&0\\\noalign{\medskip}0&0&0&0&-1
\\\noalign{\medskip}1&0&0&1&0\end {array} \right] 
$ 
$A_{45}\!=\!\! \left[ \begin {array}{ccccc} 0&0&0&-1&0\\\noalign{\medskip}0&0&0&0&0\\\noalign{\medskip}0&0&0&0&0\\\noalign{\medskip}1&0&0&0&-1
\\\noalign{\medskip}0&0&0&1&0\end {array} \right] 
$
\end{center}
\end{changemargin}

\subsection{List of Matrices 2}
For $k=4$, $N=8$ the matrix below is invertible:
\begin{align}\label{48invertible}
  \begin{bmatrix}
    \begin{array}{cccc|cccc|cccc|cccc} 0&0&0&0&0&1&1&1&0&-1&0&0&0&0&1&0\\0&0&0&0&-1&0&1&0&1&0&-1&0&0&0&1&0
\\0&0&0&0&-1&-1&0&0&0&1&0&-1&-1&-1&0&0
\\0&0&0&0&-1&0&0&0&0&0&1&0&0&0&0&0
\\\hline0&-1&-1&-1&0&0&0&0&0&0&0&0&0&0&-1&0
\\1&0&-1&0&0&0&0&0&0&0&0&1&0&0&-1&0
\\1&1&0&0&0&0&0&0&0&0&0&0&1&1&0&0
\\1&0&0&0&0&0&0&0&0&-1&0&0&0&0&0&0
\\\hline0&1&0&0&0&0&0&0&0&0&0&0&0&0&0&0
\\-1&0&1&0&0&0&0&-1&0&0&0&0&0&0&0&1
\\0&-1&0&1&0&0&0&0&0&0&0&0&0&0&0&1
\\0&0&-1&0&0&1&0&0&0&0&0&0&0&-1&-1&0
\\\hline0&0&-1&0&0&0&1&0&0&0&0&0&0&0&0&0
\\0&0&-1&0&0&0&1&0&0&0&0&-1&0&0&0&0
\\1&1&0&0&-1&-1&0&0&0&0&0&-1&0&0&0&0
\\0&0&0&0&0&0&0&0&0&1&1&0&0&0&0&0\end{array}
  \end{bmatrix}
\end{align}\\

Here we list the matrices satisfying \ref{c1},\ref{c2}, \ref{c3} and \ref{c4} for $k=3$, $N=9,10,11$ which was used in the proof of theorem \ref{thmnode}. A similar list also appears in \cite{hol}. 
\newpage
\begin{itemize}
\item $N=9$\vspace{10pt}

\begin{changemargin}{-4cm}{-4cm}
\begin{center}
$A_{12}=\left[ \begin {array}{cccccc} 0&2&3&2&1&2\\\noalign{\medskip}-2&0&1&0&3&2\\\noalign{\medskip}-3&-1&0&1&3&2\\\noalign{\medskip}-2&0&-1&0&0&0
\\\noalign{\medskip}-1&-3&-3&0&0&0\\\noalign{\medskip}-2&-2&-2&0&0&0
\end {array} \right] 
$,
$A_{13}=\left[ \begin {array}{cccccc} 0&-3&-2&-2&-3&-1\\\noalign{\medskip}3&0&0&-2&0&-1\\\noalign{\medskip}2&0&0&-1&-3&-3\\\noalign{\medskip}2&2&1&0
&0&0\\\noalign{\medskip}3&0&3&0&0&0\\\noalign{\medskip}1&1&3&0&0&0
\end {array} \right] 
$, $A_{23}=\left[ \begin {array}{cccccc} 0&2&0&0&3&3\\\noalign{\medskip}-2&0&1&1&0&2\\\noalign{\medskip}0&-1&0&2&1&0\\\noalign{\medskip}0&-1&-2&0&0&0
\\\noalign{\medskip}-3&0&-1&0&0&0\\\noalign{\medskip}-3&-2&0&0&0&0
\end {array} \right] 
$
\end{center}
\end{changemargin}
\vspace{5pt}
\hrulefill
\vspace{10pt}
\item $N=10$
\vspace{10pt}
\begin{changemargin}{-4.4cm}{-4.4cm}
\begin{center}
$\!\!\!\!\!\!\!\!A_{12}\!=\!\! \left[ \begin {array}{ccccccc} 0&1&0&0&0&2&0\\\noalign{\medskip}-1&0&
2&1&0&2&2\\\noalign{\medskip}0&-2&0&2&0&2&2\\\noalign{\medskip}0&-1&-2
&0&2&0&0\\\noalign{\medskip}0&0&0&-2&0&0&0\\\noalign{\medskip}-2&-2&-2
&0&0&0&0\\\noalign{\medskip}0&-2&-2&0&0&0&0\end {array} \right] 
$\!\!
$A_{13}\!=\!\!\left[ \begin {array}{ccccccc} 0&-2&-1&-2&-1&-1&-1\\\noalign{\medskip}2&0&-1&-2&-1&0&-1\\\noalign{\medskip}1&1&0&-1&-2&-
2&-1\\\noalign{\medskip}2&2&1&0&-1&0&-1\\\noalign{\medskip}1&1&2&1&0&0
&0\\\noalign{\medskip}1&0&2&0&0&0&0\\\noalign{\medskip}1&1&1&1&0&0&0
\end {array} \right] 
$\!\! $A_{23}\!=\!\!\left[ \begin {array}{ccccccc} 0&1&2&2&1&1&2\\\noalign{\medskip}-1&0&1&2&1&2&1\\\noalign{\medskip}-2&-1&0&2&0&1&0\\\noalign{\medskip}-2&-2&
-2&0&0&2&1\\\noalign{\medskip}-1&-1&0&0&0&0&0\\\noalign{\medskip}-1&-2
&-1&-2&0&0&0\\\noalign{\medskip}-2&-1&0&-1&0&0&0\end {array} \right] 
$
\end{center}
\end{changemargin}
\vspace{5pt}
\hrulefill
\vspace{10pt}
\item $N=11$ 
\vspace{10pt}
\begin{changemargin}{-2.5cm}{-2.5cm}
\begin{center}
$A_{12}=\left[ \begin {array}{cccccccc} 0&1&0&2&2&2&0&1\\\noalign{\medskip}-1&0&0&0&0&0&1&2\\\noalign{\medskip}0&0&0&1&2&1&1&0\\\noalign{\medskip}-
2&0&-1&0&1&0&2&0\\\noalign{\medskip}-2&0&-2&-1&0&0&2&2
\\\noalign{\medskip}-2&0&-1&0&0&0&0&0\\\noalign{\medskip}0&-1&-1&-2&-2
&0&0&0\\\noalign{\medskip}-1&-2&0&0&-2&0&0&0\end {array} \right] 
$\hspace{20pt}
$A_{13}=\left[ \begin {array}{cccccccc} 0&-1&-1&-1&-1&-2&-1&-2\\\noalign{\medskip}1&0&-1&0&-2&-2&0&0\\\noalign{\medskip}1&1&0&-1&0&0
&0&-1\\\noalign{\medskip}1&0&1&0&0&-1&-2&-2\\\noalign{\medskip}1&2&0&0
&0&0&-1&0\\\noalign{\medskip}2&2&0&1&0&0&0&0\\\noalign{\medskip}1&0&0&
2&1&0&0&0\\\noalign{\medskip}2&0&1&2&0&0&0&0\end {array} \right] 
$

$A_{23}=\left[ \begin {array}{cccccccc} 0&1&0&2&1&2&2&2\\\noalign{\medskip}-1&0&0&1&1&0&0&2\\\noalign{\medskip}0&0&0&0&1&2&2&0\\\noalign{\medskip}-
2&-1&0&0&2&0&0&2\\\noalign{\medskip}-1&-1&-1&-2&0&1&0&0
\\\noalign{\medskip}-2&0&-2&0&-1&0&0&0\\\noalign{\medskip}-2&0&-2&0&0&0
&0&0\\\noalign{\medskip}-2&-2&0&-2&0&0&0&0\end {array} \right] 
$
\end{center}
\end{changemargin}

\end{itemize}

\end{document}